\documentclass[11pt]{article}

\usepackage[a4paper,margin=1in]{geometry}
\usepackage{amsmath,amssymb,amsthm,mathtools}
\usepackage{booktabs,array,longtable}
\usepackage{enumitem}
\usepackage{microtype}
\usepackage{xcolor}
\usepackage{authblk}
\usepackage{hyperref}
\usepackage[nameinlink,capitalise,noabbrev]{cleveref}

\hypersetup{
  colorlinks=true,
  linkcolor=blue!55!black,
  citecolor=blue!55!black,
  urlcolor=blue!55!black
}

\newtheorem{theorem}{Theorem}[section]
\newtheorem{proposition}[theorem]{Proposition}
\newtheorem{lemma}[theorem]{Lemma}
\newtheorem{corollary}[theorem]{Corollary}
\theoremstyle{remark}
\newtheorem{remark}[theorem]{Remark}

\usepackage{etoolbox}
\AtBeginEnvironment{proposition}{\crefalias{theorem}{proposition}}
\AtBeginEnvironment{lemma}{\crefalias{theorem}{lemma}}
\AtBeginEnvironment{corollary}{\crefalias{theorem}{corollary}}
\AtBeginEnvironment{remark}{\crefalias{theorem}{remark}}

\numberwithin{equation}{section}

\newcommand{\one}{\mathbf 1}
\newcommand{\R}{\mathbb R}
\newcommand{\ip}[2]{\left\langle #1,#2\right\rangle}
\newcommand{\norm}[1]{\left\lVert #1\right\rVert}
\newcommand{\Tr}{\operatorname{Tr}}

\newcommand{\dd}{\,\mathrm d}
\newcommand{\E}{\mathbb E}
\newcommand{\GammaLaw}{\operatorname{Gamma}}
\newcommand{\BetaLaw}{\operatorname{Beta}}
\newcommand{\supp}{\operatorname{supp}}

\newcommand{\dens}[2]{t(#1,#2)}
\newcommand{\tC}[2]{t(C_{#1},#2)}
\newcommand{\tP}[2]{t(P_{#1},#2)}
\newcommand{\edens}[1]{t(K_{2},#1)}

\newcommand{\goodman}[2]{#1^{#2}-#1(1-#1)^{#2-1}}
\newcommand{\goal}{\goodman{p}{m}}

\newcommand{\Tker}[1]{T_{#1}}
\newcommand{\dW}{d_W}

\newcommand{\intsq}{\int_{[0,1]^2}}
\newcommand{\intcube}[1]{\int_{[0,1]^{#1}}}
\newcommand{\ddxy}{\dd x\dd y}

\newcommand{\rep}[2]{#1^{\times #2}}

\title{Goodman-type lower bounds for odd cycles I: The dense regime}
\author[1]{Seonghyuk~Im}
\author[2,3]{Taeyoung~Kim}
\author[2,3]{Hongseok~Yang}
\affil[1]{Center for AI and Natural Science, Korea Institute for Advanced Study (KIAS), Korea}
\affil[2]{School of Computing, KAIST, Korea}
\affil[3]{School of Computational Sciences, Korea Institute for Advanced Study (KIAS), Korea}
\affil[ ]{\newline\texttt{seonghyuk@kias.re.kr}, \texttt{taeyoungkim21@kaist.ac.kr}, \texttt{hongseokyang@kias.re.kr}}
\date{}

\begin{document}
\maketitle

\begin{abstract}
  Goodman proved a lower bound on triangle counts which, in graphon language, states that every graphon of edge density $p$ has triangle density at least $2p^2-p$.
  Bennett, Dudek, Lidick\'{y}, and Pikhurko established the corresponding bound for $C_5$, determining its minimum density at every Tur\'an density $p=1-1/k$ with an integer $k\ge3$.
  We extend this bound to every odd cycle in the dense regime: for every odd integer $m\ge3$ and every graphon $W$ with edge density $p\ge2/3$,
  \[
    \tC mW\ge p^m-p(1-p)^{m-1}.
  \]
  For $2/3\le p<1$, equality holds precisely for balanced complete $k$-partite graphons with an integer $k\ge3$, up to graphon equivalence, and we establish qualitative stability at these densities.
  In the same density range, we also strengthen Kim and Lee's density-sensitive commonality inequality for odd cycles by adding the nonnegative term $t(P_{m-1},1-W)-(1-p)^{m-1}$ to the lower bound, where $P_{m-1}$ is the path with $m-1$ edges. To prove the main results, we develop a spectral excursion expansion and use Dirichlet averaging to reduce the resulting polynomial inequalities in arbitrarily many spectral variables to a one-variable problem.
  We prove the one-variable inequality through a uniform inequality comparing odd and even moments of shifted Gamma random variables.
  We also formalise the odd-cycle lower-bound theorem in Lean~4.
\end{abstract}

\section{Introduction}\label{sec:intro}

Determining the (asymptotic) minimum density of a fixed graph $H$ at a prescribed edge density is a classical problem in extremal graph theory.
Here subgraph densities are normalised by counting labelled copies of $H$ and dividing by $n^{|V(H)|}$, where $n$ is the number of vertices of the host graph.
Typical examples are Tur\'an-type problems, which ask how dense a graph can be while the $H$-density remains zero.
The Erd\H os--Stone--Simonovits theorem~\cite{erdos1966limit,erdos1946structure} answers this question asymptotically: if $H$ has chromatic number $\chi(H)=r+1$, then every $H$-free graph has edge density at most $1-1/r+o(1)$.

The supersaturation theorem of Erd\H os and Simonovits~\cite{erdos1983supersaturated} further shows that exceeding the threshold $1-1/r$ by a fixed amount forces the $H$-density to be bounded away from zero.
More quantitatively, Sidorenko~\cite{sidorenko1991inequalities} conjectured that for a bipartite graph $H$, the minimum $H$-density among graphs of edge density $p$ is asymptotically attained by the random graph with edge density $p$. 
This conjecture has been proved for several classes of bipartite graphs~\cite{conlon2010approximate,conlon2018some,conlon2017finite,conlon2021sidorenko,hatami2010graph,im2026sidorenko,kim2016two,li2025logarithmic,sidorenko1993correlation,szegedy2014information}, and very recently, OpenAI announced a counterexample~\cite{openai2026counterexample}. However, it is still far from completely understood which bipartite graphs satisfy Sidorenko's conjecture.

For complete graphs, the corresponding question is the classical Erd\H os--Rademacher, or clique-density, problem~\cite{erdos1962number}.
Goodman~\cite{goodman1959sets} proved a simple lower bound on triangle counts, which implies that the asymptotic triangle density at edge density $p$ is at least $2p^2-p$.
At each density $p=1-1/k$ with an integer $k\ge2$, this bound is asymptotically attained by balanced complete $k$-partite graphs.
Between these densities, the exact asymptotic minimum follows a more intricate, piecewise-defined curve, which was determined by Razborov~\cite{razborov2008minimal} using flag algebras.
Nikiforov~\cite{nikiforov2010number} treated $K_4$, and Reiher~\cite{reiher2016clique} proved the clique density theorem for every complete graph, establishing the asymptotic minimum conjectured by Lov\'asz and Simonovits~\cite{lovasz1983number}.
The asymptotic structure of extremal and almost extremal graphs was characterised by Pikhurko and Razborov~\cite{pikhurko2017asymptotic} for triangles and by Kim, Liu, Pikhurko, and Sharifzadeh~\cite{kim2020asymptotic} for general cliques.

The case of odd cycles is less well understood than that of cliques.
By the Erd\H os--Stone--Simonovits theorem, every odd cycle has Tur\'an density $1/2$, so the asymptotic minimum odd-cycle density is zero when the edge density $p$ is at most $1/2$.
For $p>1/2$, Bennett, Dudek, Lidick\'{y}, and Pikhurko~\cite{bennett2020minimizing} determined the asymptotic minimum $C_5$-density at every density $p=1-1/k$ with an integer $k\ge3$: it equals $p^5-p(1-p)^4$, with balanced complete $k$-partite graphs attaining this value asymptotically.

Our aim is a Goodman-type polynomial lower bound that holds uniformly for all odd cycle lengths and is sharp at the same multipartite densities.
We formulate our results using graphons, which encode limits of dense graphs and express asymptotic edge and subgraph densities as analytic quantities~\cite{lovasz2012large,lovasz2006limits}.
A \emph{graphon} is a symmetric measurable function $W\colon[0,1]^2\to[0,1]$.
For a finite simple graph $H$ with vertex set $\{1,\ldots,v\}$ and edge set $E(H)$, its \emph{homomorphism density} in $W$ is
\[
  \dens HW=\intcube{v}\prod_{ij\in E(H)}W(x_i,x_j)\dd x_1\cdots\dd x_v,
\]
and the \emph{edge density} of $W$ is $\edens W=\intsq W(x,y)\ddxy$.
These densities are the limits of the normalised labelled copy counts along convergent sequences of finite graphs.
Consequently, asymptotic lower bounds on subgraph densities in finite graphs become exact inequalities for graphons.

For an integer $k\ge2$, let $W_k$ denote the balanced complete $k$-partite graphon: it has $k$ parts of equal measure, takes the value $1$ between distinct parts, and takes the value $0$ within each part.
Its edge density is $p=1-1/k$, and for every odd $m\ge3$, the density of the $m$-cycle $C_m$ in $W_k$ is given by
\begin{equation}\label{eq:intro-kpartite-value}
  \tC m{W_k}=\goal.
\end{equation}
This suggests a common form for the triangle and five-cycle bounds.
Our main result extends this bound to every odd cycle in the dense regime $p\ge2/3$.
\begin{theorem}[Odd-cycle Goodman bound for $p\ge2/3$]\label{thm:main}
  Let $m\ge3$ be odd, and let $W$ be a graphon with $\edens W=p\ge2/3$.
  Then,
  \begin{equation}\label{eq:intro-main}
    \tC mW\ge\goal.
  \end{equation}
  Equality holds if and only if $W$ is equal to $W_k$ up to relabelling for some integer $k\ge3$, or $W\equiv1$.
\end{theorem}

At each density $p=1-1/k$, $k\ge3$, this uniqueness also yields a stability result in cut distance (\Cref{cor:stability}).

A companion paper proves the odd-cycle Goodman bound \eqref{eq:intro-main} in the remaining nontrivial range $1/2<p<2/3$ using different methods.
For $p\le1/2$, the bound is immediate because $p^m-p(1-p)^{m-1}$ is nonpositive.

\subsection{Extension of commonality}
Our proof of \Cref{thm:main} passes through a stronger inequality involving $W$ and its complement $1-W$.
For a graph $H$, we say that $H$ is \emph{common} if the random red-blue edge colouring of $E(K_n)$ asymptotically minimises the number of monochromatic copies of $H$.
In the graphon setting, this notion translates to the inequality
\[
  \dens HW+\dens H{1-W}\ge2^{1-e(H)}
\]
for every graphon $W$.
The notion of common graphs has been extensively studied in extremal combinatorics. 
Goodman~\cite{goodman1959sets} proved that the triangle is common.
A famous conjecture by Erd\H{o}s~\cite{erdos1962number} states that all complete graphs are common, and Burr and Rosta~\cite{burr1980ramsey} strengthened it by conjecturing that all graphs are common.
This was later disproved by Thomason~\cite{thomason1989disproof}, who showed that $K_4$ is not common.
Many other graphs have been studied to determine whether or not they are common~\cite{grzesik2022tripartite,hatami2012non,jagger1996multiplicities,ko2023common,kral2025common,lee2026disconnected,sidorenko1989cycles}, but the complete classification remains open.

Kim and Lee~\cite{kim2024extended} established the following density-sensitive strengthening for every path or cycle $H$ and every graphon\footnote{Their result holds more generally for every kernel, that is, every bounded symmetric measurable function $W\colon[0,1]^2\to\mathbb{R}$.} $W$ with $\edens W=p$:
\begin{equation}\label{eq:intro-extended-commonality}
  \dens HW+\dens H{1-W}\ge p^{e(H)}+(1-p)^{e(H)}.
\end{equation}
This inequality proves the conjecture that every odd cycle is strongly common, posed by Behague, Morrison, and Noel in the first arXiv version of their paper~\cite[Conjecture~9.7]{behague2022commonv1}.
Since $p^{e(H)}+(1-p)^{e(H)}\ge2^{1-e(H)}$, \eqref{eq:intro-extended-commonality} implies commonality.

For odd cycles with $p\ge2/3$, our second main theorem strengthens \eqref{eq:intro-extended-commonality} by a path-density term.
A graphon $W$ is called $p$-regular if $\int_0^1 W(x,y)\dd y=p$ for almost every $x$.

\begin{theorem}\label{thm:commonality}
  Let $m\ge3$ be odd, and let $W$ be a graphon with $\edens W=p\ge2/3$.
  Then, with $P_{m-1}$ the path with $m-1$ edges,
  \begin{equation}\label{eq:intro-commonality}
    \tC mW+\tC m{1-W}\ge p^m+(1-p)^m+\bigl[\tP{m-1}{1-W}-(1-p)^{m-1}\bigr].
  \end{equation}
  Furthermore, equality holds if and only if $W$ is $p$-regular.
\end{theorem}

The Blakley--Roy inequality~\cite{blakley1965holder} (independently proved by Mulholland and Smith~\cite{mulholland1959inequality}) gives $\tP{m-1}{1-W}\ge(1-p)^{m-1}$.
Thus, \eqref{eq:intro-commonality} strengthens \eqref{eq:intro-extended-commonality} for odd cycles.

The inequality in \Cref{thm:main} follows immediately from \Cref{thm:commonality}.
Since $0\le1-W\le1$, deleting one edge from the $m$-cycle gives
\begin{equation}\label{eq:cycle-path-comparison}
  \tC m{1-W}\le\tP{m-1}{1-W}.
\end{equation}
Rearranging \eqref{eq:intro-commonality} and using $(1-p)^m-(1-p)^{m-1}=-p(1-p)^{m-1}$, we obtain
\[
  \tC mW-\bigl(\goal\bigr)\ge\tP{m-1}{1-W}-\tC m{1-W}\ge0.
\]
Equality in \eqref{eq:intro-main} holds precisely when equality holds in both \eqref{eq:intro-commonality} and \eqref{eq:cycle-path-comparison}.
By \Cref{thm:commonality}, this means that $W$ is $p$-regular and $\tC m{1-W}=\tP{m-1}{1-W}$.
For $p<1$, these conditions force $1-W$ to be a union of equal cliques, as proved in \Cref{prop:main-equality}, giving the balanced complete multipartite graphons $W_k$ with $k\ge3$.
At $p=1$, $W\equiv1$ almost everywhere and equality holds; by \eqref{eq:intro-kpartite-value}, the converse holds for $W_k$.

We also remark that for $m\in\{3,5,7\}$, the explicit calculations in \Cref{sec:short-cycles} prove the inequality in \Cref{thm:commonality} for all $p\ge1/2$.
Therefore, the lower bound in \Cref{thm:main} holds for every $p\in[0,1]$ for $C_3$, $C_5$, and $C_7$.
On the other hand, there is a counterexample for $C_9$ when $p = 501/1000$, showing that \Cref{thm:commonality} does not hold for all odd cycles and all $p$. Thus, to prove \Cref{thm:main} for the full range of $p\in[0,1]$, one needs a different approach.
See \Cref{sec:conclusion} for the detailed discussion.

\subsection{Proof strategy: four reductions}
Our proof strategy can be summarised as four reductions that simplify the problem while preserving the essential structure of the inequality.
Before introducing the reductions, we set up some preliminary observations.

Let $\Tker W$ denote the integral operator with kernel $W$, formally defined in \Cref{subsec:graphon-preliminaries}.
The trace formula expresses the two cycle densities as
\[
  \tC mW=\Tr(\Tker W^m),
  \qquad
  \tC m{1-W}=\Tr(\Tker{1-W}^m).
\]
Each trace is the sum of the $m$th powers of the corresponding operator's eigenvalues.
Our aim is to show that the sum of these two traces exceeds $p^m+(1-p)^m$ by at least the path-density correction in \eqref{eq:intro-commonality}.
We obtain this bound through four reductions, all uniform in the odd cycle length $m$.

\paragraph{Step 1: Trace decomposition.}
We split the space into constant and zero-average functions and expand the two cycle traces with respect to this decomposition.
For a real-valued function $f\in L^2[0,1]$, write
\[
  f=c\one+f_0,\qquad
  c=\int_0^1 f(x)\dd x,\qquad
  \int_0^1 f_0(x)\dd x=0,
\]
where $\one$ is the constant function equal to one.
Define the operators
\[
  \Pi_{\mathrm c} f=c\one,\qquad \Pi_0 f=f_0,
\]
which retain the constant and zero-average components, respectively.
For either $T=\Tker W$ or $T=\Tker{1-W}$, the trace decomposes as
\[
    \Tr(T^m)
    =  \Tr\bigl((\Pi_{\mathrm c}T)^m\bigr)
       +\Tr\bigl((\Pi_0T)^m\bigr)
       +\sum_{\substack{(s_1,\ldots,s_m)\in\{\mathrm c,0\}^m\\
                         \text{not all }s_i\text{ equal}}}
       \Tr\bigl(\Pi_{s_1}T\cdots\Pi_{s_m}T\bigr).
\]
The first term retains only the constant component, the second retains only the zero-average component, and the final sum includes every sequence using both components.
For $\Tker W$ and $\Tker{1-W}$, the first contributions sum to $p^m+(1-p)^m$, while the second contributions cancel for odd $m$.
It therefore remains to show that the two sums involving both components together are at least the path-density term.
We handle these mixed trace terms by approximating $W$ in $L^1$ by step graphons and using finite-matrix calculations (\Cref{lem:step-reduction}).

\paragraph{Step 2: Excursion expansion.}
We reduce this comparison to the nonnegativity of polynomials.
Grouping the mixed trace terms into excursions gives the following expression for the difference between the two sides of \eqref{eq:intro-commonality}:
\[
  \begin{aligned}
    & \Bigl(\tC mW+\tC m{1-W}-p^m-(1-p)^m\Bigr)
      - \Bigl[\tP{m-1}{1-W}-(1-p)^{m-1}\Bigr]\\
    &\quad {} =
      \Biggl(\sum_{\substack{(s_1,\ldots,s_m)\in\{\mathrm c,0\}^m\\
                            \text{not all }s_i\text{ equal}}}
      \Tr\bigl(\Pi_{s_1}\Tker W\cdots\Pi_{s_m}\Tker W\bigr)
      +\sum_{\substack{(s_1,\ldots,s_m)\in\{\mathrm c,0\}^m\\
                             \text{not all }s_i\text{ equal}}}
      \Tr\bigl(\Pi_{s_1}\Tker{1-W}\cdots\Pi_{s_m}\Tker{1-W}\bigr)\Biggr)\\
    &\quad\qquad\qquad {} - \Bigl[\tP{m-1}{1-W}-(1-p)^{m-1}\Bigr]\\
    &\quad {} =
      \sum_{r=1}^{(m-1)/2}
      \int_{[-1/2,1/2]^r}
      P_{m,r}(1-p;\lambda_1,\ldots,\lambda_r)
      \prod_{i=1}^r\dd\mu(\lambda_i).
  \end{aligned}
\]
The first equality is the trace decomposition from the first reduction.
The second expresses the remaining trace sums, after subtracting the path-density correction, in terms of explicit polynomials.
Here $\mu$ is a finite nonnegative measure determined by $W$, supported on $[-1/2,1/2]$, and constructed in \Cref{subsec:spectral-setup}; the $P_{m,r}$ are symmetric polynomials defined in \eqref{eq:Pmr}.
In each mixed trace term, viewed cyclically, the zero-average projections form blocks separated by constant projections.
Each such block is called an \emph{excursion}, and $r$ counts the number of excursions.
The measure supplies nonnegative weights, so it suffices to prove, for every $1\le r\le(m-1)/2$, that
\begin{equation}\label{eq:intro-polynomial-positive}
  P_{m,r}(1-p;\lambda_1,\ldots,\lambda_r)\ge0
  \qquad\text{for every }(\lambda_1,\ldots,\lambda_r)\in[-1/2,1/2]^r.
\end{equation}

\paragraph{Step 3: Dirichlet averaging.}
For $m\in\{3,5,7\}$, direct computation proves \eqref{eq:intro-polynomial-positive} for all $p\ge1/2$ (\Cref{sec:short-cycles}).
For the general case, we reduce \eqref{eq:intro-polynomial-positive} to a one-variable inequality.
Choose nonnegative weights $\Theta_1,\ldots,\Theta_r$, summing to one, uniformly at random.
These weights have the Dirichlet$(1,\ldots,1)$ distribution.
Set $L=\sum_{i=1}^r\Theta_i\lambda_i$.
\Cref{lem:pge23-dirichlet} gives
\[
  P_{m,r}(1-p;\lambda_1,\ldots,\lambda_r)
  =\E\bigl[P_{m,r}(1-p;L,\ldots,L)\bigr].
\]
Since $L\in[-1/2,1/2]$ whenever all the $\lambda_i$ lie in this interval, it suffices to prove
\begin{equation}\label{eq:intro-diagonal-positive}
  P_{m,r}(1-p;\lambda,\ldots,\lambda)\ge0
  \qquad\text{for every }\lambda\in[-1/2,1/2].
\end{equation}
Although the number $r$ of variables can grow with $m$, \eqref{eq:intro-diagonal-positive} is a polynomial inequality in just one variable for each choice of $m$, $r$, and $p$.

\paragraph{Step 4: Gamma moment bound.}
We prove \eqref{eq:intro-diagonal-positive} using an inequality
between odd and even moments of $Y-b$, where $Y$ has a Gamma distribution
and $b\ge0$ is a real shift.
We write $\GammaLaw(a,1)$ for the Gamma distribution with shape $a>0$
and rate $1$, whose density is $y^{a-1}e^{-y}/\Gamma(a)$ for $y>0$.
Here $\Gamma(a)=\int_0^\infty y^{a-1}e^{-y}\dd y$ is the Gamma function.
An integral representation of $P_{m,r}(1-p;\lambda,\ldots,\lambda)$ establishes \eqref{eq:intro-diagonal-positive} when $\lambda\le0$ (\Cref{lem:pge23-ell-negative}).
For $\lambda>0$, let $Y\sim\GammaLaw(r,1)$ and
$Z\sim\GammaLaw(m-r,1)$ be independent.
By \eqref{eq:pge23-gamma-expectation},
\[
  P_{m,r}(1-p;\lambda,\ldots,\lambda)
  =\binom{m-r-1}{r}
    \E\!\left[\rho_{m-2r,m}\!\left(1-p+\frac{\lambda}{Z}Y\right)\right],
\]
where $\rho_{m-2r,m}$ is the polynomial defined in \eqref{eq:pge23-rho}.
To prove \eqref{eq:intro-diagonal-positive} for $\lambda>0$, it suffices to show that, for each fixed $z>0$, the expectation over $Y$ satisfies
\begin{equation}\label{eq:intro-expectation-positive}
  \E_Y\!\left[\rho_{m-2r,m}\!\left(1-p+\frac{\lambda}{z}Y\right)\right]\ge0.
\end{equation}
For all integers $r,j\ge1$ and all real shifts $b\ge0$, we prove the odd-to-even moment bound
\begin{equation}\label{eq:intro-moment-bound}
  3jb\,\E(Y-b)^{2j-1}\le(r+j)\,\E(Y-b)^{2j}.
\end{equation}
This bound yields \eqref{eq:intro-expectation-positive} for $p\ge2/3$,
as shown in the proof of \Cref{prop:pge23-gamma-positive}.
We state its rescaled form in \Cref{lem:pge23-gamma-moment} and prove it in
\Cref{app:gamma-moment}, using a differential equation satisfied by the even moments.

The application of the moment bound in \Cref{prop:pge23-gamma-positive} is the only step requiring $p\ge2/3$.
Averaging over $Z$ then proves \eqref{eq:intro-diagonal-positive} for $\lambda>0$, completing the fourth reduction.
\Cref{lem:pge23-dirichlet} extends this nonnegativity to the multivariable polynomials in \eqref{eq:intro-polynomial-positive}.
These polynomials are integrated against nonnegative weights in the second step, so the required trace comparison follows, proving \Cref{thm:commonality}.
The deduction preceding this overview then gives the odd-cycle bound in \Cref{thm:main}.

\paragraph{Equality and stability.}
We bound the gap in the strengthened commonality inequality from below in terms of the degree variance, so equality forces regularity.
For $p<1$, equality in the comparison between cycle and path densities then forces the complement to be a union of equal cliques, identifying the balanced complete multipartite graphons as the unique extremisers at the corresponding densities.
Qualitative stability follows from this uniqueness, compactness in cut distance, and continuity of homomorphism densities.

\paragraph{Organisation.}
\Cref{sec:preliminaries} collects graphon preliminaries and justifies the reduction to step graphons.
\Cref{sec:matrix-reduction} develops the finite-matrix representation and introduces the spectral quantities used in the proof.
\Cref{sec:excursions} develops the excursion expansion and reduces \Cref{thm:commonality} to the inequalities $P_{m,r}\ge0$.
\Cref{sec:short-cycles} works out this reduction for $C_3$, $C_5$, and $C_7$, while \Cref{sec:excursion-positivity} proves the inequalities for every odd cycle.
\Cref{sec:equality-stability} identifies the equality cases and deduces qualitative stability.
We close with further questions in \Cref{sec:conclusion}.

\section{Graphon preliminaries}
\label{sec:preliminaries}

\subsection{Graphon preliminaries and operators}
\label{subsec:graphon-preliminaries}

For an integrable function $F\colon[0,1]^2\to\R$, its \emph{cut norm} is
\[
  \norm F_\square=\sup_{S,T\subseteq[0,1]}
    \left|\int_{S\times T}F(x,y)\ddxy\right|,
\]
where the supremum is over measurable sets $S,T\subseteq[0,1]$.

For graphons $W,V$, their \emph{cut distance} is
\begin{equation}\label{eq:cut-distance}
  \delta_\square(W,V)=\inf_\varphi\norm{W-V^\varphi}_\square,
\end{equation}
where $\varphi$ ranges over measure-preserving bijections of $[0,1]$ modulo null sets, and $V^\varphi(x,y)=V(\varphi(x),\varphi(y))$.
The cut distance is a pseudometric on graphons.
Graphons at zero cut distance are called \emph{equivalent}, and identifying equivalent graphons gives a metric space.
This space is compact, and for every fixed finite simple graph $H$, the map $W\mapsto\dens HW$ is continuous in this metric~\cite{borgs2008convergent,lovasz2012large}.
In particular, edge, cycle, and path densities are continuous in cut distance.

For a graphon $W$, define the integral operator $\Tker W$ on $L^2[0,1]$ by
\[
  (\Tker Wf)(x)=\int_0^1 W(x,y)f(y)\dd y.
\]
Since $W$ is symmetric and square-integrable, $\Tker W$ is a compact self-adjoint Hilbert--Schmidt operator.
Its nonzero eigenvalues $(\lambda_i)_i$, counted with multiplicity, satisfy
\begin{equation}\label{eq:graphon-hilbert-schmidt}
  \sum_i\lambda_i^2=\intsq W(x,y)^2\ddxy<\infty.
\end{equation}
The eigenvalues are bounded, so their power sums converge absolutely for every integer $m\ge3$.
Thus, $\Tker W^m$ is trace class, and the cycle trace formula~\cite{lovasz2012large} gives
\begin{equation}\label{eq:graphon-cycle-trace}
  \tC mW=\Tr(\Tker W^m)=\sum_i\lambda_i^m.
\end{equation}
For the path $P_\ell$ with $\ell\ge1$ edges, repeated integration gives
\begin{equation}\label{eq:graphon-path-operator}
  \tP\ell W=\ip\one{\Tker W^\ell\one},
\end{equation}
where $\one$ is the constant function equal to one.

We note that in the next section, we reduce the problem to the finite matrix calculation so the main inequalities can be verified without using this operator framework, but later stability results will use this.

\subsection{Step-graphon approximation}
\label{subsec:step-approximation}

To prove the inequalities in \Cref{thm:main,thm:commonality} by finite matrix calculations, we first restrict attention to \emph{step graphons}: graphons that are constant on each product of two parts of a finite measurable partition of $[0,1]$.
We will use partitions into intervals of equal length, so these products form a grid of equally sized rectangles.
Each such graphon is specified by the finite matrix of its values, and its cycle and path densities can be expressed through matrix products.
The following lemma justifies this restriction: proving the two inequalities in those theorems for all such graphons, with no bound on the number of parts, proves them for arbitrary graphons as well.

\begin{lemma}
\label{lem:step-reduction}
  It suffices to prove the inequalities in \Cref{thm:main,thm:commonality} for graphons that
  are constant on each rectangle $I_a\times I_b$, where
  $I_1,\ldots,I_N$ partition $[0,1]$ into intervals of length $1/N$
  and $N$ ranges over all positive integers.
\end{lemma}
Starting from an arbitrary graphon $W$, divide $[0,1]$ into $N$ equal intervals and replace $W$ on each rectangle $I_a\times I_b$ by its average value there.
This produces a step graphon with exactly the same edge density as $W$.
Starting with two equal intervals and repeatedly halving each partition interval gives increasingly fine grids, with $N=2,4,8,\ldots$ parts.
Averaging $W$ on each of these grids produces a sequence of step graphons whose cycle and path densities, and those of their complements, converge to the corresponding densities of $W$ and $1-W$.
Thus, if the two inequalities hold for every graphon in this sequence, they also hold for $W$ by taking limits.
The formal proof is given in \Cref{app:step-reduction}.

\section{Finite-matrix setup and spectral measure}
\label{sec:matrix-reduction}

This section prepares the finite-matrix calculation underlying the proof.
We express cycle and path densities as traces and matrix products, then introduce the spectral quantities used in the excursion expansion in the next section.

By \Cref{lem:step-reduction}, we may fix a step graphon $W$ with edge density $\edens W=p\ge2/3$, constant on each rectangle $I_a\times I_b$, where $I_1,\ldots,I_N$ partition $[0,1]$ into intervals of length $1/N$.
Write
\[
  U=1-W,
  \qquad
  q=1-p.
\]
Thus, $U$ is the complementary graphon and has edge density $q$.
The assumption $p\ge2/3$ will be needed only for the positivity argument in \Cref{sec:excursion-positivity}; the matrix identities and excursion expansions developed before that point hold at every edge density.

\subsection{Matrix representation}
\label{subsec:step-matrices}

A \emph{step kernel} on this partition is a symmetric function $K\colon[0,1]^2\to\R$ that is constant on each rectangle $I_a\times I_b$.
It is a step graphon when its values lie in $[0,1]$; in particular, $W$, $U$, and the constant-one graphon are step kernels on this partition.
Write $K_{ab}$ for the value of $K$ on $I_a\times I_b$, and define the $N\times N$ matrix
$M_K$ and the vector $e\in\R^N$ by
\begin{equation*}
  M_K=\left(\frac{K_{ab}}N\right)_{a,b=1}^N,
  \qquad
  e=\frac1{\sqrt N}(1,\ldots,1)^{\mathsf T}.
\end{equation*}
Then, $M_K$ is symmetric and $\norm e=1$. For a step graphon $K$, this normalisation expresses the relevant graph densities through matrix powers:
\begin{align}
  \Tr(M_K^j)
  &=\frac1{N^j}\sum_{a_1,\ldots,a_j=1}^N
    K_{a_1a_2}\cdots K_{a_ja_1}
   =\dens{C_j}K \qquad (j\ge3),\notag\\
  e^{\mathsf T}M_K^ke
  &=\frac1{N^{k+1}}\sum_{a_0,\ldots,a_k=1}^N
    K_{a_0a_1}\cdots K_{a_{k-1}a_k}
   =\dens{P_k}K \qquad (k\ge1),\notag\\
  e^{\mathsf T}M_Ke
  &=\frac1{N^2}\sum_{a,b=1}^NK_{ab}
  =\edens K. \label{eq:edge-matrix}
\end{align}

\subsection{Complementary blocks and spectral measure}
\label{subsec:spectral-setup}

We now separate the constant and zero-average components of $M_W$ and $M_U$, and construct a nonnegative measure $\mu$ whose moments represent the contributions of excursions between these components.

Let $\mathcal E_0=e^\perp$, and let $I$ be the $N\times N$ identity
matrix. The matrices
\[
  P=ee^{\mathsf T},
  \qquad
  Q=I-P
\]
are the orthogonal projections of $\R^N$ onto $\R e$ and $\mathcal E_0$,
respectively. They are the matrix counterparts of the projections
$\Pi_{\mathrm c}$ and $\Pi_0$ used in the introduction.
Relative to the decomposition $\R e\oplus\mathcal E_0$, these projections have the block forms
\begin{equation}\label{eq:block-PQ}
  P=\begin{pmatrix}1&0\\0&0\end{pmatrix},
  \qquad
  Q=\begin{pmatrix}0&0\\0&I_{\mathcal E_0}\end{pmatrix},
\end{equation}
where $I_{\mathcal E_0}$ is the identity on $\mathcal E_0$.

Define the vector $g\in\R^N$ and the symmetric linear map
$A\colon\mathcal E_0\to\mathcal E_0$ by
\begin{equation*}
  g=M_Ue-qe,
  \qquad
  A=QM_UQ\big|_{\mathcal E_0}.
\end{equation*}
By \eqref{eq:edge-matrix},
\[
  e^{\mathsf T}g=e^{\mathsf T}M_Ue-q=0,
\]
so $g\in\mathcal E_0$.
The constant-one graphon has matrix $M_1=P$, and hence $M_W=P-M_U$.
Therefore, relative to $\R e\oplus\mathcal E_0$,
\begin{equation}\label{eq:block-UW}
  M_U=
  \begin{pmatrix}
    q&g^{\mathsf T}\\
    g&A
  \end{pmatrix},
  \qquad
  M_W=
  \begin{pmatrix}
    p&-g^{\mathsf T}\\
    -g&-A
  \end{pmatrix}.
\end{equation}

\begin{lemma}\label{lem:half-norm}
  Every eigenvalue of $A$ lies in $[-1/2,1/2]$.
\end{lemma}

\begin{proof}
  Throughout this proof, vector coordinates and matrix entries refer to the standard basis of $\R^N$, indexed by the partition intervals $I_1,\ldots,I_N$.
  In these coordinates,
  \[
    0\le (M_U)_{ab}=\frac{U_{ab}}{N}\le\frac1N=(M_1)_{ab},
    \qquad 1\le a,b\le N,
  \]
  so $0\le M_U\le M_1$ entrywise.
  Let $f\in\mathcal E_0$ with $\norm f=1$, and define its coordinatewise
  positive and negative parts by
  \[
    (f_+)_i=\max\{f_i,0\},
    \qquad
    (f_-)_i=\max\{-f_i,0\}.
  \]
  Both vectors have nonnegative entries, $f=f_+-f_-$, and $|f|=f_++f_-$.
  Since $e^{\mathsf T}f=0$,
  \[
    e^{\mathsf T}f_+=e^{\mathsf T}f_-=:a,
    \qquad
    2a=e^{\mathsf T}|f|\le\norm e\norm f=1.
  \]
  Expanding $f=f_+-f_-$ gives
  \[
    f^{\mathsf T}M_Uf
    =f_+^{\mathsf T}M_Uf_+
     -f_+^{\mathsf T}M_Uf_-
     -f_-^{\mathsf T}M_Uf_+
     +f_-^{\mathsf T}M_Uf_-.
  \]
  Since $f_+$, $f_-$, and $M_U$ have nonnegative entries,
  $f_+^{\mathsf T}M_Uf_-$ and $f_-^{\mathsf T}M_Uf_+$ are nonnegative.
  Dropping these terms gives
  \[
    f^{\mathsf T}M_Uf
    \le f_+^{\mathsf T}M_Uf_++f_-^{\mathsf T}M_Uf_-
      \le2a^2\le\frac12,
  \]
  where the second inequality uses the entrywise comparison $M_U\le M_1$
  in the standard coordinates, the nonnegative entries of $f_+$ and $f_-$,
  and the identities $f_+^{\mathsf T}M_1f_+=f_-^{\mathsf T}M_1f_-=a^2$.  Similarly,
  $f_+^{\mathsf T}M_Uf_+$ and $f_-^{\mathsf T}M_Uf_-$ are nonnegative,
  so dropping these terms gives
  \[
    f^{\mathsf T}M_Uf
    \ge-f_+^{\mathsf T}M_Uf_- - f_-^{\mathsf T}M_Uf_+
      \ge-2a^2\ge-\frac12.
  \]
  Here the second inequality again uses the entrywise comparison $M_U\le M_1$
  in the standard coordinates and the nonnegative entries of $f_+$ and $f_-$,
  together with $f_+^{\mathsf T}M_1f_-=f_-^{\mathsf T}M_1f_+=a^2$.
  Since $f^{\mathsf T}Af=f^{\mathsf T}M_Uf$, the Rayleigh--Ritz
  characterisation of the eigenvalues of the symmetric map $A$ proves the
  claim.
\end{proof}

Let $\lambda_1,\ldots,\lambda_{N-1}$ be the eigenvalues of $A$, counted
with multiplicity, and choose corresponding orthonormal eigenvectors
$u_1,\ldots,u_{N-1}$ in $\mathcal E_0$. For each $i$, let $P_i$ be the
orthogonal projection onto the one-dimensional subspace spanned by $u_i$.
Define the finite nonnegative measure
\begin{equation*}
  \mu=\sum_{i=1}^{N-1}\norm{P_ig}^2\delta_{\lambda_i}.
\end{equation*}
Here $\delta_{\lambda_i}$ is the unit point mass at $\lambda_i$.

\begin{lemma}\label{lem:spectral-moments}
  The measure $\mu$ is independent of the choice of orthonormal eigenvectors
  within each eigenspace. For every integer $j\ge0$,
  \begin{equation}\label{eq:moments}
    \supp\mu\subseteq[-1/2,1/2],
    \qquad
    g^{\mathsf T}A^jg=\int\lambda^j\dd\mu(\lambda),
    \qquad
    \mu(\R)=\norm g^2.
  \end{equation}
\end{lemma}

\begin{proof}
  By the spectral decomposition of $A$,
  \[
    g^{\mathsf T}A^jg
    =\sum_{i=1}^{N-1}\lambda_i^j\norm{P_ig}^2
    =\int\lambda^j\dd\mu(\lambda).
  \]
  Taking $j=0$ gives the total mass identity, and the support bound follows from \Cref{lem:half-norm}.
  The sum of weights at each eigenvalue is the squared norm of the projection of $g$ onto its eigenspace, so $\mu$ is independent of the eigenbasis.
\end{proof}

Thus, the quantities $g^{\mathsf T}A^jg$ are moments of a single nonnegative measure $\mu$ supported on $[-1/2,1/2]$.
In the next section, we use this representation to express the difference between the two sides of the inequality \eqref{eq:intro-commonality} in \Cref{thm:commonality} as a sum of repeated integrals of polynomials against $\mu$.
Since $\mu$ is nonnegative, proving that these polynomials are nonnegative will establish the theorem.

\section{Excursion expansions}
\label{sec:excursions}

Using the matrix representation and spectral measure from \Cref{sec:matrix-reduction}, we express the difference between the two sides of \eqref{eq:intro-commonality} as polynomial integrals against $\mu$ (\Cref{lem:polynomial-reduction}).
Since $\mu$ is nonnegative, the nonnegativity of these polynomials, proved in \Cref{sec:short-cycles,sec:excursion-positivity}, will establish \Cref{thm:commonality}.

To obtain these formulas, we expand matrix powers by inserting $P+Q=I$.
For an odd integer $m\ge3$, the cycle density of $U$ satisfies
\begin{equation}\label{eq:trace-projection-expansion}
  \tC mU=\Tr(M_U^m)
  =\sum_{R_1,\ldots,R_m\in\{P,Q\}}
    \Tr(R_1M_UR_2M_U\cdots R_mM_U).
\end{equation}
Each sequence $R_1\cdots R_m$ of letters $P$ and $Q$, called a \emph{word}, corresponds to exactly one summand in \eqref{eq:trace-projection-expansion}.
We call a word \emph{mixed} if it contains both $P$ and $Q$.
For trace terms, we regard the word as cyclic, so its last and first letters are consecutive.

A \emph{block} is a maximal consecutive run of the same letter.
Thus, in a mixed word, each $Q$-block has a $P$ immediately before and after it.
For example, the segment
\[
  P\,\underbrace{QQ}_{\text{one $Q$-block}}\,P
\]
of a word corresponds to the matrix segment $P M_U Q M_U Q M_U P$.
In a mixed cyclic word, we call each $Q$-block an \emph{excursion}.
For a $Q$-block of length $\beta+1$, where $\beta\ge0$ is an integer, the corresponding matrix segment satisfies
\[
  P M_U(Q M_U)^{\beta+1}P=(g^{\mathsf T}A^\beta g)P.
\]
By \eqref{eq:moments}, its scalar coefficient is $\int\lambda^\beta\dd\mu(\lambda)$.
Together with the powers of $q$ from the $P$-blocks, these moments turn the sums over block lengths into polynomial integrals.

We group words by their number $r$ of excursions and their block lengths.
The resulting monomial sums are complete homogeneous symmetric polynomials, defined as follows.
For integers $k\ge1$ and $d\ge0$ and variables $y_1,\ldots,y_k$, the
complete homogeneous symmetric polynomial of degree $d$ is
\begin{equation}\label{eq:h-def}
  h_d(y_1,\ldots,y_k)
  =\sum_{\substack{\gamma_1,\ldots,\gamma_k\ge0\\
                    \gamma_1+\cdots+\gamma_k=d}}
   y_1^{\gamma_1}\cdots y_k^{\gamma_k},
  \qquad h_0=1.
\end{equation}
We write $\rep ak$ in an argument list when $a$ is repeated $k$ times.
In \eqref{eq:h-def}, the exponents $\gamma_i$ are nonnegative integers.
All repeated integrals against $\mu$ below are over $[-1/2,1/2]$,
which contains its support, unless another domain is specified.

We now make this expansion precise for cycle densities.

\begin{lemma}
\label{lem:cycle-excursions}
  For every odd $m\ge3$,
  \begin{align}
    \tC mU
    ={}&q^m+\Tr(A^m)
    +\sum_{r=1}^{(m-1)/2}\frac mr
      \int\!\cdots\!\int
      h_{m-2r}(\rep qr,\lambda_1,\ldots,\lambda_r)
      \prod_{i=1}^r\dd\mu(\lambda_i),
    \label{eq:cycle-U}\\
    \tC mW
    ={}&p^m-\Tr(A^m)
    +\sum_{r=1}^{(m-1)/2}\frac mr
      \int\!\cdots\!\int
      h_{m-2r}(\rep pr,-\lambda_1,\ldots,-\lambda_r)
      \prod_{i=1}^r\dd\mu(\lambda_i).
    \label{eq:cycle-W}
  \end{align}
\end{lemma}

\begin{proof}
  We first prove \eqref{eq:cycle-U}.
  In \eqref{eq:trace-projection-expansion}, the word consisting only of $P$ gives $q^m$, while the word consisting only of $Q$ gives $\Tr(A^m)$.

  We now consider the remaining words, which contain both $P$ and $Q$.
  For a mixed word with $r$ cyclic $Q$-blocks, choose a boundary where $Q$ is followed by $P$ and rotate the word to start with $P$ and end with $Q$, which does not change the trace.
  For $1\le i\le r$, write the lengths of the $i$th $P$-block and the following $Q$-block as $\alpha_i+1$ and $\beta_i+1$, respectively, with $\alpha_i,\beta_i\ge0$.
  To evaluate the corresponding trace, use the block forms in \eqref{eq:block-PQ} and \eqref{eq:block-UW} to obtain
  \[
    Q M_U=\begin{pmatrix}0&0\\g&A\end{pmatrix},
    \qquad
    (Q M_U)^{\beta_i+1}P
    =\begin{pmatrix}0&0\\A^{\beta_i}g&0\end{pmatrix}.
  \]
  The second identity follows by induction on $\beta_i$: right-multiplication by $P$ gives the case $\beta_i=0$, and left-multiplication by $Q M_U$ multiplies the lower-left block by $A$.
  The block forms also give $P M_U P=qP$, and left-multiplication of the second identity by $P M_U$ yields
  \[
    P M_U(Q M_U)^{\beta_i+1}P
    =(g^{\mathsf T}A^{\beta_i}g)P.
  \]
  Using $P M_U P=qP$, the pair of blocks, ending at the next $P$, satisfies
  \[
    (P M_U)^{\alpha_i+1}(Q M_U)^{\beta_i+1}P
    =q^{\alpha_i}(g^{\mathsf T}A^{\beta_i}g)P.
  \]
  By cyclicity and $P^2=P$, appending a final $P$ does not change the trace.
  Multiplying the block factors and using $\Tr(P)=1$ gives the length condition and trace formula
  \begin{equation}\label{eq:excursion-length-weight}
    \sum_{i=1}^r(\alpha_i+\beta_i)=m-2r,
    \qquad
    \Tr(R_1M_UR_2M_U\cdots R_mM_U)
    =q^{\alpha_1+\cdots+\alpha_r}
     \prod_{i=1}^rg^{\mathsf T}A^{\beta_i}g.
  \end{equation}

  We count pairs consisting of a mixed word with positions labelled $1,\ldots,m$ and $r$ cyclic $Q$-blocks, together with a marked boundary where $Q$ is followed by $P$.
  Each word has exactly $r$ choices of the mark.
  Conversely, an ordered tuple $(\alpha_1,\beta_1,\ldots,\alpha_r,\beta_r)$ of nonnegative integers satisfying the length condition in \eqref{eq:excursion-length-weight}, together with the labelled position $j$ of the first $P$ after the mark, determines exactly one marked word: place the successive $P$- and $Q$-blocks of the specified lengths cyclically from $j$.
  Each tuple therefore gives $m$ marked words, distinct even for periodic words because their marks occupy different labelled positions.
  These two constructions are mutually inverse, and by cyclicity and \eqref{eq:excursion-length-weight}, the trace weight depends only on the tuple.
  Summing the weights in these two ways gives
  \begin{equation}\label{eq:U-r-excursions}
      \Bigg(
        r\!\!\!\!
        \sum_{\substack{R_1\cdots R_m\in\{P,Q\}^m\\
                       \text{both $P$ and $Q$ occur}\\
                       R_1\cdots R_m\text{ has }r\text{ cyclic $Q$-blocks}}}
        \Tr(R_1M_U\cdots R_mM_U)\Bigg)
      =
      \Bigg(
        m \!\!\!\!
        \sum_{\substack{\alpha_1,\beta_1,\ldots,
                          \alpha_r,\beta_r\ge0\\
                          \sum_{i=1}^r(\alpha_i+\beta_i)=m-2r}}
        q^{\alpha_1+\cdots+\alpha_r}
        \prod_{i=1}^rg^{\mathsf T}A^{\beta_i}g\Bigg).
  \end{equation}
  Dividing \eqref{eq:U-r-excursions} by $r$ gives the factor $m/r$.
  By \eqref{eq:moments} and the definition \eqref{eq:h-def},
  \begin{align*}
    &\sum_{\substack{\alpha_1,\beta_1,\ldots,\alpha_r,\beta_r\ge0\\
                     \sum_{i=1}^r(\alpha_i+\beta_i)=m-2r}}
      q^{\alpha_1+\cdots+\alpha_r}
      \prod_{i=1}^rg^{\mathsf T}A^{\beta_i}g\\
    &\quad=\int\!\cdots\!\int
      \sum_{\substack{\alpha_1,\beta_1,\ldots,\alpha_r,\beta_r\ge0\\
                     \sum_{i=1}^r(\alpha_i+\beta_i)=m-2r}}
      q^{\alpha_1+\cdots+\alpha_r}
      \lambda_1^{\beta_1}\cdots\lambda_r^{\beta_r}
      \prod_{i=1}^r\dd\mu(\lambda_i)\\
    &\quad=\int\!\cdots\!\int
      h_{m-2r}(\rep qr,\lambda_1,\ldots,\lambda_r)
      \prod_{i=1}^r\dd\mu(\lambda_i).
  \end{align*}
  Summing over $r$ and adding $q^m+\Tr(A^m)$ proves \eqref{eq:cycle-U}.

  For $M_W$, \eqref{eq:block-UW} replaces
  $(q,g,A)$ by $(p,-g,-A)$.  The two crossing signs cancel in each block,
  while
  \[
    (-g)^{\mathsf T}(-A)^{\beta_i}(-g)
    =(-1)^{\beta_i}g^{\mathsf T}A^{\beta_i}g.
  \]
  Hence $q$ becomes $p$, every $\lambda_i$ becomes $-\lambda_i$, and
  $\Tr(A^m)$ becomes $-\Tr(A^m)$.  This proves \eqref{eq:cycle-W}.
\end{proof}

We next derive the corresponding formula for path densities, which supplies the correction term in \eqref{eq:intro-commonality}.
Let $P_k$ denote the path with $k$ edges, with $P_0$ a single vertex.
Recall from \Cref{subsec:step-matrices} that
\begin{equation*}
  \tP kU=e^{\mathsf T}M_U^ke.
\end{equation*}

\begin{lemma}
\label{lem:path-excursions}
  For every $k\ge0$,
  \begin{equation}\label{eq:path-expansion}
    \tP kU=q^k+
    \sum_{r=1}^{\lfloor k/2\rfloor}
      \int\!\cdots\!\int
      h_{k-2r}(\rep q{(r+1)},\lambda_1,\ldots,\lambda_r)
      \prod_{i=1}^r\dd\mu(\lambda_i).
  \end{equation}
\end{lemma}

\begin{proof}
  For $k=0$ and $k=1$, the formula gives $\tP0U=e^{\mathsf T}e=1$ and
  $\tP1U=e^{\mathsf T}M_Ue=q$, respectively, with an empty sum over $r$.
  Now let $k\ge2$. Since $Pe=e$ and $e^{\mathsf T}P=e^{\mathsf T}$,
  we have $\tP kU = e^{\mathsf T}PM_U^kPe$.
  Inserting $P+Q=I$ between its $k$ copies of $M_U$ and expanding gives
  \begin{equation}\label{eq:path-projection-expansion}
    \tP kU=
    \sum_{(R_1,\ldots,R_{k-1})\in\{P,Q\}^{k-1}}
      e^{\mathsf T}P M_U R_1M_U\cdots R_{k-1}M_UPe.
  \end{equation}
  Associate each summand in \eqref{eq:path-projection-expansion} with
  the word $PR_1\cdots R_{k-1}P$.
  This word has $k+1$ letters, separated by the $k$ copies of $M_U$.
  We read it from its fixed left endpoint to its fixed right endpoint;
  its first and last letters are separate occurrences of $P$.
  If every letter is $P$, the identity $P M_U P=qP$ gives the summand
  $q^k e^{\mathsf T}Pe=q^k$.

  Now fix the number $r\ge1$ of $Q$-blocks in a word.
  Because the word starts and ends with $P$, it has $r+1$ $P$-blocks,
  alternating with the $r$ $Q$-blocks. Let
  $\alpha_0+1,\ldots,\alpha_r+1$ be the numbers of letters in the
  successive $P$-blocks, and let $\beta_1+1,\ldots,\beta_r+1$ be the
  numbers of letters in the successive $Q$-blocks.  For example, the
  following word has $r=2$:
  \[
    \underbrace{PP}_{\alpha_0=1}\,
    \underbrace{QQ}_{\beta_1=1}\,
    \underbrace{PP}_{\alpha_1=1}\,
    \underbrace{Q}_{\beta_2=0}\,
    \underbrace{PP}_{\alpha_2=1}.
  \]
  Counting the $k+1$ letters gives
  \[
    \sum_{i=0}^r(\alpha_i+1)+\sum_{i=1}^r(\beta_i+1)=k+1,
    \qquad
    \sum_{i=0}^r\alpha_i+\sum_{i=1}^r\beta_i=k-2r.
  \]
  In particular, $r\le\lfloor k/2\rfloor$.
  To evaluate the summand in \eqref{eq:path-projection-expansion}
  associated with this word, use the same identities as in the proof
  of \Cref{lem:cycle-excursions}: each internal step in a $P$-block
  contributes $q$, and the excursion through the $i$th $Q$-block
  contributes $g^{\mathsf T}A^{\beta_i}g$.
  Collapsing these blocks leaves the product of their scalar factors
  times $P$. Since $e^{\mathsf T}Pe=1$, the summand is
  \[
    e^{\mathsf T}P M_U R_1M_U\cdots R_{k-1}M_UPe
    =q^{\alpha_0+\cdots+\alpha_r}
      \prod_{i=1}^rg^{\mathsf T}A^{\beta_i}g.
  \]

  With fixed endpoints, each admissible ordered tuple
  $(\alpha_0,\beta_1,\alpha_1,\ldots,\beta_r,\alpha_r)$ corresponds to exactly one word.
  Combining this correspondence with \eqref{eq:moments} and the definition \eqref{eq:h-def} gives
  \begin{align}
    &\sum_{\substack{(R_1,\ldots,R_{k-1})\in\{P,Q\}^{k-1}\\
                     PR_1\cdots R_{k-1}P
                     \text{ has }r\text{ $Q$-blocks}}}
      e^{\mathsf T}P M_U R_1M_U\cdots R_{k-1}M_UPe\notag\\
    &\quad=\sum_{\substack{\alpha_0,\ldots,\alpha_r,\beta_1,\ldots,\beta_r\ge0\\
                         \sum_{i=0}^r\alpha_i+
                         \sum_{i=1}^r\beta_i=k-2r}}
      q^{\alpha_0+\cdots+\alpha_r}
      \prod_{i=1}^rg^{\mathsf T}A^{\beta_i}g\notag\\
    &\quad=\int\!\cdots\!\int
      \sum_{\substack{\alpha_0,\ldots,\alpha_r,\beta_1,\ldots,\beta_r\ge0\\
                      \sum_{i=0}^r\alpha_i+
                      \sum_{i=1}^r\beta_i=k-2r}}
      q^{\alpha_0+\cdots+\alpha_r}
      \lambda_1^{\beta_1}\cdots\lambda_r^{\beta_r}
      \prod_{i=1}^r\dd\mu(\lambda_i)\notag\\
    &\quad=\int\!\cdots\!\int
      h_{k-2r}(\rep q{(r+1)},\lambda_1,\ldots,\lambda_r)
      \prod_{i=1}^r\dd\mu(\lambda_i).
    \label{eq:path-r-excursions}
  \end{align}
  The $r+1$ exponents $\alpha_0,\ldots,\alpha_r$ account for the $r+1$ copies of $q$; together with $\beta_1,\ldots,\beta_r$, their total degree is $k-2r$.
  Summing \eqref{eq:path-r-excursions} over $r$ and including the word consisting only of $P$ proves \eqref{eq:path-expansion}.
\end{proof}

We now combine the cycle and path formulas to express the difference between the two sides of \eqref{eq:intro-commonality} as polynomial integrals against $\mu$.
The following lemma completes the reduction announced at the start of this section.

\begin{lemma}[Reduction to excursion polynomials]
\label{lem:polynomial-reduction}
  Let $m\ge3$ be odd, and write $p=1-q$. For each integer
  $1\le r\le(m-1)/2$, let
  \begin{equation}\label{eq:pge23-n}
    n=m-2r,
  \end{equation}
  which is a positive odd integer, and define the excursion polynomial,
  symmetric in $\lambda_1,\ldots,\lambda_r$, by
  \begin{align}
    P_{m,r}(q;\lambda_1,\ldots,\lambda_r)
    ={}&\frac mr\Bigl[
      h_n(\rep pr,-\lambda_1,\ldots,-\lambda_r)
      +h_n(\rep qr,\lambda_1,\ldots,\lambda_r)
    \Bigr]\notag\\
    &-h_{n-1}(\rep q{(r+1)},\lambda_1,\ldots,\lambda_r).
    \label{eq:Pmr}
  \end{align}
  Then the sum of integrals
  \begin{equation}\label{eq:Phi-expansion}
    \Phi_m:=\sum_{r=1}^{(m-1)/2}
      \int\!\cdots\!\int
      P_{m,r}(q;\lambda_1,\ldots,\lambda_r)
      \prod_{i=1}^r\dd\mu(\lambda_i)
  \end{equation}
  satisfies
  \[
    \Phi_m=\tC mW+\tC mU-p^m-q^m
    -\bigl[\tP{m-1}U-q^{m-1}\bigr].
  \]
\end{lemma}

\begin{proof}
  Adding the cycle expansions \eqref{eq:cycle-U} and
  \eqref{eq:cycle-W} cancels $\Tr(A^m)$ with $-\Tr(A^m)$ and gives the
  following identity, with $n=m-2r$ in each summand:
  \[
    \begin{aligned}
      \tC mW+\tC mU
      ={}p^m+q^m
      &+\sum_{r=1}^{(m-1)/2}\frac mr
      \int\!\cdots\!\int
        h_n(\rep pr,-\lambda_1,\ldots,-\lambda_r)
      \prod_{i=1}^r\dd\mu(\lambda_i)\\
      &+\sum_{r=1}^{(m-1)/2}\frac mr
      \int\!\cdots\!\int
        h_n(\rep qr,\lambda_1,\ldots,\lambda_r)
      \prod_{i=1}^r\dd\mu(\lambda_i).
    \end{aligned}
  \]
  For $k=m-1$, the path expansion \eqref{eq:path-expansion} uses the
  same range $1\le r\le(m-1)/2$, and its polynomial degree is
  $(m-1)-2r=n-1$. Thus,
  \[
    \tP{m-1}U=q^{m-1}+
    \sum_{r=1}^{(m-1)/2}
      \int\!\cdots\!\int
      h_{n-1}(\rep q{(r+1)},\lambda_1,\ldots,\lambda_r)
      \prod_{i=1}^r\dd\mu(\lambda_i).
  \]
  For each fixed number $r$ of excursions, the two cycle integrals and
  the path integral use the same product measure
  $\prod_{i=1}^r\dd\mu(\lambda_i)$.
  Combining the two cycle contributions, each with coefficient $m/r$, and subtracting the path contribution gives the integrand $P_{m,r}(q;\lambda_1,\ldots,\lambda_r)$ by the definition \eqref{eq:Pmr}.
  Therefore,
  \begin{equation}\label{eq:commonality-gap}
    \tC mW+\tC mU-\tP{m-1}U
    =p^m+q^m-q^{m-1}+\Phi_m.
  \end{equation}
  By rearranging \eqref{eq:commonality-gap}, we obtain the desired expression for $\Phi_m$.
\end{proof}

Since $\mu$ is nonnegative and supported on $[-1/2,1/2]$, it suffices to
prove, for every $1\le r\le(m-1)/2$, that
\[
  P_{m,r}(q;\lambda_1,\ldots,\lambda_r)\ge0
  \qquad\text{for all }(\lambda_1,\ldots,\lambda_r)\in[-1/2,1/2]^r,
\]
where $q=1-p\in[0,1/3]$.
We prove the nonnegativity of these polynomials in \Cref{sec:short-cycles,sec:excursion-positivity}.

\section{The cycles \texorpdfstring{$C_3,C_5,C_7$}{C3, C5, C7}}
\label{sec:short-cycles}

In this section, we explicitly compute the excursion polynomials for $C_3$, $C_5$, and $C_7$.
The resulting positivity proves the inequalities in \Cref{thm:main,thm:commonality} for all $p\ge1/2$; we give this deduction together with the proof of the main inequalities at the end of \Cref{sec:excursion-positivity}.

\begin{proposition}\label{prop:short-cycles}
  Suppose that $0\le q\le1/2$ and set $p=1-q$. For all $m\in\{3,5,7\}$,
  $1\le r\le(m-1)/2$, and $\lambda_1,\ldots,\lambda_r\in\R$, we have
  \[
    P_{m,r}(q;\lambda_1,\ldots,\lambda_r)\ge0.
  \]
\end{proposition}

\begin{proof}
  For $m=3$, only $r=1$ occurs, and
  \[
    \begin{aligned}
      P_{3,1}(q;\lambda)
      &=3\bigl[h_1(p,-\lambda)+h_1(q,\lambda)\bigr]
        -h_0(q,q,\lambda)\\
      &=3\bigl[(p-\lambda)+(q+\lambda)\bigr]-1=2.
    \end{aligned}
  \]

  For $m=5$ and $r=1$, we have $n=3$, and hence
  \[
    \begin{aligned}
      P_{5,1}(q;\lambda)
      ={}&5\bigl[h_3(p,-\lambda)+h_3(q,\lambda)\bigr]
          -h_2(q,q,\lambda)\\
      ={}&5\bigl[(p^3+q^3)+(q^2-p^2)\lambda
                  +(p+q)\lambda^2\bigr]
         -(3q^2+2q\lambda+\lambda^2)\\
      ={}&5\bigl[(1-3q+3q^2)+(2q-1)\lambda+\lambda^2\bigr]
         -(3q^2+2q\lambda+\lambda^2)\\
      ={}&4\lambda^2+(8q-5)\lambda+12q^2-15q+5.
    \end{aligned}
  \]
  Completing squares gives
  \[
    P_{5,1}(q;\lambda)
    =4\left(\lambda+q-\frac58\right)^2
     +8\left(q-\frac58\right)^2+\frac5{16}>0.
  \]
  For $r=2$, we have $n=1$, so
  \[
    \begin{aligned}
      P_{5,2}(q;\lambda_1,\lambda_2)
      ={}&\frac52\bigl[
        h_1(p,p,-\lambda_1,-\lambda_2)
        +h_1(q,q,\lambda_1,\lambda_2)\bigr]\\
        &-h_0(q,q,q,\lambda_1,\lambda_2)\\
      ={}&\frac52\bigl[(2p-\lambda_1-\lambda_2)
                       +(2q+\lambda_1+\lambda_2)\bigr]-1=4.
    \end{aligned}
  \]
  For $m=7$ and $r=1$, we have $n=5$.  Therefore,
  \[
    \begin{aligned}
      P_{7,1}(q;\lambda)
      ={}&7\bigl[h_5(p,-\lambda)+h_5(q,\lambda)\bigr]
          -h_4(q,q,\lambda)\\
      ={}&7\bigl[p^5+q^5+(q^4-p^4)\lambda
          +(p^3+q^3)\lambda^2+(q^2-p^2)\lambda^3
          +(p+q)\lambda^4\bigr]\\
         &-(5q^4+4q^3\lambda+3q^2\lambda^2
             +2q\lambda^3+\lambda^4)\\
      ={}&7\bigl[(1-5q+10q^2-10q^3+5q^4)
          +(4q^3-6q^2+4q-1)\lambda\\
         &\qquad +(1-3q+3q^2)\lambda^2
          +(2q-1)\lambda^3+\lambda^4\bigr]\\
         &-(5q^4+4q^3\lambda+3q^2\lambda^2
             +2q\lambda^3+\lambda^4)\\
      ={}&6\lambda^4+(12q-7)\lambda^3
          +(18q^2-21q+7)\lambda^2\\
         &+(24q^3-42q^2+28q-7)\lambda
          +30q^4-70q^3+70q^2-35q+7.
    \end{aligned}
  \]
  Define
  \[
    \begin{aligned}
      D(q)&=288q^2-336q+119,\\
      C(q)&=24q^3-42q^2+28q-7,\\
      N(q)&=5184q^6-18144q^5+28602q^4-25802q^3 +13874q^2-4165q+539.
    \end{aligned}
  \]
  Direct expansion gives
  \[
    P_{7,1}(q;\lambda)=
      6\left(\lambda^2+\left(q-\frac7{12}\right)\lambda\right)^2+\frac{D(q)}{24}
      \left(\lambda+\frac{12C(q)}{D(q)}\right)^2
      +\frac{N(q)}{D(q)}.
  \]
  With $y=1/2-q\ge0$,
  \[
    D(q)=288y^2+48y+23>0
  \]
  and
  \[
    8N(q)=41472y^6+20736y^5+21456y^4+7984y^3+2032y^2+316y+11>0.
  \]
  Hence $P_{7,1}(q;\lambda)>0$.

  For $r=2$, we have $n=3$, and
  \[
    \begin{aligned}
      P_{7,2}(q;\lambda_1,\lambda_2)
      ={}&\frac72\bigl[
        h_3(p,p,-\lambda_1,-\lambda_2)
        +h_3(q,q,\lambda_1,\lambda_2)\bigr]\\
        &-h_2(q,q,q,\lambda_1,\lambda_2)\\
      ={}&\frac72\bigl[4(p^3+q^3)
        +3(q^2-p^2)(\lambda_1+\lambda_2)
        +2(\lambda_1^2+\lambda_1\lambda_2+\lambda_2^2)\bigr]\\
        &-\bigl[6q^2+3q(\lambda_1+\lambda_2)
          +(\lambda_1^2+\lambda_1\lambda_2+\lambda_2^2)\bigr]\\
      ={}&6(\lambda_1^2+\lambda_1\lambda_2+\lambda_2^2)
        +\left(18q-\frac{21}{2}\right)(\lambda_1+\lambda_2)\\
        &+36q^2-42q+14.
    \end{aligned}
  \]
  Completing squares gives
  \[
      P_{7,2}(q;\lambda_1,\lambda_2)=
      \left(\frac32(\lambda_1-\lambda_2)^2
      +\frac92\left(\lambda_1+\lambda_2+2q-\frac76\right)^2
      +18\left(q-\frac7{12}\right)^2+\frac74\right)>0.
  \]
  Finally, for $r=3$, we have $n=1$, so
  \[
    \begin{aligned}
      P_{7,3}(q;\lambda_1,\lambda_2,\lambda_3)
      ={}&\frac73\bigl[
        h_1(p,p,p,-\lambda_1,-\lambda_2,-\lambda_3)
        +h_1(q,q,q,\lambda_1,\lambda_2,\lambda_3)\bigr]-1\\
      ={}&\frac73\bigl[(3p-\lambda_1-\lambda_2-\lambda_3)
        +(3q+\lambda_1+\lambda_2+\lambda_3)\bigr]-1
      \\
      ={}&6.
    \end{aligned}
  \]
This concludes the proof.
\end{proof}

\section{Positivity of the excursion polynomials}
\label{sec:excursion-positivity}

We now complete the proof of the strengthened extended commonality inequality
in \Cref{thm:commonality} and hence the odd-cycle Goodman bound in
\Cref{thm:main}.
The excursion expansion in \Cref{sec:excursions} reduces this task to showing that
\begin{equation}\label{eq:pge23-positivity-goal}
  P_{m,r}(q;\lambda_1,\ldots,\lambda_r)\ge0
\end{equation}
whenever $0\le q\le1/3$, $m\ge3$ is odd, $1\le r\le(m-1)/2$, and each
$\lambda_i\in[-1/2,1/2]$.
We prove \eqref{eq:pge23-positivity-goal} in two steps.
First, Dirichlet averaging reduces the inequality to the case in which all
spectral variables $\lambda_1,\ldots,\lambda_r$ equal a common value $\lambda$.
Second, we prove that the resulting diagonal polynomial is nonnegative,
beginning with a representation through a Beta expectation.
This representation gives nonnegativity directly when $\lambda\le0$.
For $\lambda>0$, a further integral representation reduces the problem to
comparing odd and even moments of $zY-1$, where $Y\sim\GammaLaw(r,1)$
and $z\ge0$.
Here, for $a>0$, $\GammaLaw(a,1)$ denotes the Gamma distribution with shape $a$ and rate $1$.
The odd-to-even Gamma moment bound in \Cref{lem:pge23-gamma-moment},
proved in \Cref{app:gamma-moment}, supplies the required comparison.
Together, the diagonal reduction and the proof of diagonal positivity establish
\eqref{eq:pge23-positivity-goal}, completing the proofs of
\Cref{thm:commonality,thm:main}.

\subsection{Dirichlet averaging and diagonal reduction}
\label{subsec:dirichlet-reduction}

We begin with an averaging identity for complete homogeneous symmetric
polynomials. It will be used both in the diagonal reduction and in the
Beta representation.

\begin{lemma}\label{lem:pge23-dirichlet-moment}
  Let $k\ge1$ and $j\ge0$ be integers, let $x_1,\ldots,x_k\in\R$, and
  let $(\Theta_1,\ldots,\Theta_k)$ be uniformly distributed on the simplex
  \[
    \Theta_i\ge0,\qquad \sum_{i=1}^k\Theta_i=1.
  \]
  Equivalently, this vector has the $\operatorname{Dirichlet}(1,\ldots,1)$
  distribution.
  Then the complete homogeneous symmetric polynomial $h_j$, defined in
  \eqref{eq:h-def}, admits the representation
  \begin{equation}\label{eq:pge23-dirichlet-moment}
    h_j(x_1,\ldots,x_k)
    =\binom{j+k-1}{k-1}
    \E\left[\left(\sum_{i=1}^k\Theta_i x_i\right)^j\right].
  \end{equation}
\end{lemma}

\begin{proof}
  Write $S=\sum_{i=1}^k\Theta_i x_i$. The multinomial theorem gives
  \[
    S^j=\Bigl(\sum_{i=1}^k\Theta_i x_i\Bigr)^j
    =\sum_{\alpha_1+\cdots+\alpha_k=j}
    \frac{j!}{\alpha_1!\cdots\alpha_k!}
    \prod_{i=1}^k\Theta_i^{\alpha_i}x_i^{\alpha_i},
  \]
  where the sum is over nonnegative integer exponents.
  The Dirichlet moments satisfy
  \[
    \E\left[\prod_{i=1}^k\Theta_i^{\alpha_i}\right]
    =\frac{(k-1)!\,\alpha_1!\cdots\alpha_k!}{(j+k-1)!},
    \qquad \alpha_1+\cdots+\alpha_k=j.
  \]
  Taking expectations therefore gives
  \begin{align*}
    \E\left[S^j\right]
    &=\sum_{\alpha_1+\cdots+\alpha_k=j}
    \frac{j!}{\alpha_1!\cdots\alpha_k!}
    \cdot\frac{(k-1)!\,\alpha_1!\cdots\alpha_k!}{(j+k-1)!}
    \prod_{i=1}^k x_i^{\alpha_i}\\
    &=\frac{j!\,(k-1)!}{(j+k-1)!}
    \sum_{\alpha_1+\cdots+\alpha_k=j}\prod_{i=1}^k x_i^{\alpha_i}
    =\frac{h_j(x_1,\ldots,x_k)}{\binom{j+k-1}{k-1}}.
  \end{align*}
  Multiplying by $\binom{j+k-1}{k-1}$ proves
  \eqref{eq:pge23-dirichlet-moment}.
\end{proof}

Fix an odd integer $m\ge3$ and $1\le r\le(m-1)/2$. Then, $n=m-2r$ is a positive odd integer.
To carry out the first step, define the diagonal polynomial
\begin{equation}\label{eq:pge23-Ptilde}
  \widetilde P_{m,r}(q,\lambda)
  =P_{m,r}(q;\lambda,\ldots,\lambda).
\end{equation}

\begin{lemma}[Diagonal reduction by Dirichlet averaging]\label{lem:pge23-dirichlet}
  Let $q,\lambda_1,\ldots,\lambda_r\in\R$, with $p=1-q$ in
  the definition \eqref{eq:Pmr} of $P_{m,r}$.
  Let $(\Theta_1,\ldots,\Theta_r)$ be uniformly distributed on the simplex
  \[
    \Theta_i\ge0,\qquad \sum_{i=1}^r\Theta_i=1.
  \]
  Then,
  \begin{equation}\label{eq:pge23-dirichlet-identity}
    P_{m,r}(q;\lambda_1,\ldots,\lambda_r)
    =\E\left[\widetilde P_{m,r}
    \left(q,\sum_{i=1}^r\Theta_i\lambda_i\right)\right].
  \end{equation}
  Consequently,
  \begin{equation*}
    \min_{\lambda_i\in[-1/2,1/2]}
    P_{m,r}(q;\lambda_1,\ldots,\lambda_r)
    =\min_{\lambda\in[-1/2,1/2]}\widetilde P_{m,r}(q,\lambda).
  \end{equation*}
\end{lemma}

\begin{proof}
  Write $S=\sum_{i=1}^r\Theta_i\lambda_i$.
  Setting all $r$ variables equal in the definition \eqref{eq:h-def}
  of the complete homogeneous symmetric polynomial $h_j$ gives
  \begin{equation}\label{eq:pge23-h-diagonal}
    h_j(\lambda,\ldots,\lambda)=\binom{j+r-1}{r-1}\lambda^j,
  \end{equation}
  the binomial coefficient being the number of monomials of degree $j$ in $r$ variables.

  It remains to express the terms of $P_{m,r}$ through
  $h_j(\lambda_1,\ldots,\lambda_r)$.
  The generating function of the $h_d$ is
  \begin{equation}\label{eq:pge23-h-generating}
    \sum_{d\ge0}h_d(y_1,\ldots,y_k)z^d
    =\prod_{i=1}^k\frac{1}{1-y_iz}.
  \end{equation}
  Thus, the $k$ equal variables and the remaining ones give separate factors:
  \[
    \begin{aligned}
      \sum_{d\ge0}h_d(a^{\times k},b_1,\ldots,b_r)z^d
      &=\frac{1}{(1-az)^k}\prod_{i=1}^r\frac{1}{1-b_iz}\\
      &=\Bigl(\sum_{s\ge0}\binom{s+k-1}{k-1}a^sz^s\Bigr)
        \Bigl(\sum_{j\ge0}h_j(b_1,\ldots,b_r)z^j\Bigr).
    \end{aligned}
  \]
  Comparing the coefficients of $z^d$ gives
  \begin{equation}\label{eq:pge23-h-split}
    h_d(a^{\times k},b_1,\ldots,b_r)
    =\sum_{j=0}^d
    \binom{d-j+k-1}{k-1}a^{d-j}h_j(b_1,\ldots,b_r).
  \end{equation}
  Apply this identity to the three terms of $P_{m,r}$: to the first two
  with $k=r$ and $a=p$ or $q$, and to the third with $k=r+1$ and $a=q$.
  Since $h_j$ is homogeneous of degree $j$,
  $h_j(-\lambda_1,\ldots,-\lambda_r)=(-1)^jh_j(\lambda_1,\ldots,\lambda_r)$.
  Therefore,
  \begin{align}
    P_{m,r}(q;\lambda_1,\ldots,\lambda_r)
    ={}&\frac {m}{r}\sum_{j=0}^{n}\binom{n-j+r-1}{r-1}
    \bigl((-1)^jp^{n-j}+q^{n-j}\bigr)
    h_j(\lambda_1,\ldots,\lambda_r) \notag\\
    &-\sum_{j=0}^{n-1}\binom{n-1-j+r}{r}q^{n-1-j}
    h_j(\lambda_1,\ldots,\lambda_r). \label{eq:pge23-P-in-h}
  \end{align}

  Write the expansion \eqref{eq:pge23-P-in-h} as
  \[
    P_{m,r}(q;\lambda_1,\ldots,\lambda_r)
    =\sum_{j=0}^nc_j\,h_j(\lambda_1,\ldots,\lambda_r),
  \]
  where $c_j$ denotes the combined coefficient of $h_j$; it does not
  depend on the $\lambda_i$.
  Setting $\lambda_1=\cdots=\lambda_r=\lambda$ and using
  \eqref{eq:pge23-h-diagonal},
  \begin{equation*}
    \widetilde P_{m,r}(q,\lambda)
    =\sum_{j=0}^nc_j\binom{j+r-1}{r-1}\lambda^j.
  \end{equation*}
  On the other hand, \Cref{lem:pge23-dirichlet-moment}, applied with
  $k=r$ and $x_i=\lambda_i$, gives
  \begin{align*}
    P_{m,r}(q;\lambda_1,\ldots,\lambda_r)
    &=\sum_{j=0}^nc_j\binom{j+r-1}{r-1}\E\left[S^j\right]\\
    &=\E\left[\sum_{j=0}^nc_j\binom{j+r-1}{r-1}S^j\right]\\
    &=\E\left[\widetilde P_{m,r}(q,S)\right],
  \end{align*}
  the last equality being $\widetilde P_{m,r}$ with $\lambda$ replaced by $S$.

  For the second claim, suppose that the $\lambda_i$ lie in
  $[-1/2,1/2]$. Then, every convex combination $S$ lies in the same
  interval, so the right-hand side of
  \eqref{eq:pge23-dirichlet-identity} is at least
  $\min_{\lambda\in[-1/2,1/2]}\widetilde P_{m,r}(q,\lambda)$.
  Taking the minimum over $\lambda_1,\ldots,\lambda_r$ shows that the
  minimum of $P_{m,r}$ on $[-1/2,1/2]^r$ is at least this diagonal minimum.
  Conversely, setting $\lambda_1=\cdots=\lambda_r=\lambda$ attains each
  diagonal value, so the minimum of $P_{m,r}$ on $[-1/2,1/2]^r$ is at most
  the diagonal minimum. The two minima are therefore equal.
\end{proof}

By \Cref{lem:pge23-dirichlet}, it now suffices to prove
\begin{equation*}
  \widetilde P_{m,r}(q,\lambda)\ge0
\end{equation*}
for
\[
  0\le q\le\frac{1}{3},
  \quad 1\le r\le\frac{m-1}{2},
  \quad -\frac{1}{2}\le\lambda\le\frac{1}{2}.
\]

\subsection{Beta representation and the nonpositive case}
\label{subsec:beta-representation}

We next obtain a Beta representation for complete homogeneous symmetric
polynomials whose variables take two repeated values. This follows from the
Dirichlet representation by grouping the weights corresponding to each value.
For $r\ge1$, the notation $X\sim\BetaLaw(r,r)$ means that $X$ has probability density function
\[
  \frac{\Gamma(2r)}{\Gamma(r)^2}x^{r-1}(1-x)^{r-1},
  \qquad 0<x<1,
\]
and zero outside this interval.

\begin{lemma}\label{lem:pge23-beta-representation}
  Let $r\ge1$ and $j\ge0$ be integers, let $a,b\in\R$, and let
  $X\sim\BetaLaw(r,r)$. Then
  \begin{equation}\label{eq:pge23-beta-h}
    h_j(a^{\times r},b^{\times r})
    =\binom{j+2r-1}{2r-1}
    \E\left[\bigl(aX+b(1-X)\bigr)^j\right].
  \end{equation}
\end{lemma}

\begin{proof}
  Let $(\Theta_1,\ldots,\Theta_{2r})$ have the
  $\operatorname{Dirichlet}(1,\ldots,1)$ distribution.
  Recall that this means the vector is uniformly distributed on the simplex
  \[
    \Theta_i\ge0\quad(1\le i\le2r),\qquad
    \sum_{i=1}^{2r}\Theta_i=1.
  \]
  The aggregation property of the Dirichlet distribution says that grouping
  coordinates adds their parameters. Thus, the pair
  $\bigl(\sum_{i=1}^r\Theta_i,\sum_{i=r+1}^{2r}\Theta_i\bigr)$ has the
  $\operatorname{Dirichlet}(r,r)$ distribution.
  A vector with this distribution is distributed as $(X,1-X)$ with
  $X\sim\BetaLaw(r,r)$, so we may realise $X$ as $\sum_{i=1}^r\Theta_i$.
  Apply \Cref{lem:pge23-dirichlet-moment} with $k=2r$, taking $x_i=a$
  for $1\le i\le r$ and $x_i=b$ for $r<i\le2r$. This gives
  \begin{align*}
    h_j(a^{\times r},b^{\times r})
    &=\binom{j+2r-1}{2r-1}
      \E\left[\left(a\sum_{i=1}^r\Theta_i+
        b\sum_{i=r+1}^{2r}\Theta_i\right)^j\right]\\
    &=\binom{j+2r-1}{2r-1}\E\left[\bigl(aX+b(1-X)\bigr)^j\right],
  \end{align*}
  which proves \eqref{eq:pge23-beta-h}.
\end{proof}

We now apply this representation with $j=n=m-2r$ to express the diagonal
polynomial as a Beta expectation.

\begin{lemma}[Beta formula for the diagonal polynomial]\label{lem:pge23-beta}
  Let $m\ge3$ be odd, let $r$ be an integer satisfying
  $1\le r\le(m-1)/2$, and set $n=m-2r$.
  Let $X\sim\BetaLaw(r,r)$ and define
  \begin{equation}\label{eq:pge23-rho}
    \rho_{n,m}(u)
    =\frac {m}{n}\bigl(u^n+(1-u)^n\bigr)-u^{n-1},
    \qquad u\in\R.
  \end{equation}
  Then, for every $q,\lambda\in\R$,
  \begin{equation}\label{eq:pge23-beta-second}
    \widetilde P_{m,r}(q,\lambda)
    =\binom{n+2r-1}{2r-1}\frac {n}{r}
    \E\left[
    X^n\rho_{n,m}\left(q+\lambda\frac{1-X}{X}\right)
    \right].
  \end{equation}
  If $\lambda>0$, then
  \begin{equation}\label{eq:pge23-beta-integral}
    \widetilde P_{m,r}(q,\lambda)
    =\frac{\Gamma(m)}{r\Gamma(n)\Gamma(r)^2}\,
    \lambda^{n+r}
    \int_0^\infty
    \frac{s^{r-1}}{(\lambda+s)^m}\rho_{n,m}(q+s)\dd s.
  \end{equation}
\end{lemma}

\begin{proof}
  By \Cref{lem:pge23-beta-representation} with $j=n$,
  \[
    h_n(a^{\times r},b^{\times r})
    =\binom{n+2r-1}{2r-1}\E\left[\bigl(aX+b(1-X)\bigr)^n\right].
  \]
  Applying \eqref{eq:pge23-h-generating} with $r$ copies each of $a$ and $b$ and differentiating with respect to $a$ gives
  \[
    \frac{\partial}{\partial a}
    \sum_{d\ge0}h_d(a^{\times r},b^{\times r})z^d
    =rz(1-az)^{-(r+1)}(1-bz)^{-r}.
  \]
  Comparing coefficients of $z^n$ yields
  \[
    \frac{\partial}{\partial a}h_n(a^{\times r},b^{\times r})
    =r\,h_{n-1}(a^{\times(r+1)},b^{\times r}).
  \]
  The expectation $\E\left[\bigl(aX+b(1-X)\bigr)^n\right]$ is a polynomial in $a,b$, since $X$ is bounded and its distribution does not depend on these parameters.
  Differentiating term by term gives
  \[
    \frac{\partial}{\partial a}\E\left[\bigl(aX+b(1-X)\bigr)^n\right]
    =n\,\E\left[X\bigl(aX+b(1-X)\bigr)^{n-1}\right].
  \]
  Combining the two derivative identities gives
  \begin{equation}\label{eq:pge23-beta-h-deriv}
    h_{n-1}(a^{\times(r+1)},b^{\times r})
    =\binom{n+2r-1}{2r-1}\frac {n}{r}
    \E\left[X\bigl(aX+b(1-X)\bigr)^{n-1}\right].
  \end{equation}

  Set $p=1-q$ and
  \[
    V_X=qX+\lambda(1-X),
    \qquad
    W_X=pX-\lambda(1-X),
    \qquad
    V_X+W_X=X.
  \]
  Apply \Cref{lem:pge23-beta-representation} to the two cycle terms
  in the definition \eqref{eq:Pmr} of $P_{m,r}$, with
  $\lambda_1=\cdots=\lambda_r=\lambda$, and
  \eqref{eq:pge23-beta-h-deriv} to the path term. This gives
  \begin{equation}\label{eq:pge23-beta-first}
    \widetilde P_{m,r}(q,\lambda)
    =\binom{n+2r-1}{2r-1}\frac {n}{r}
    \left[
      \frac {m}{n}\E\left[V_X^n+W_X^n\right]
      -\E\left[XV_X^{n-1}\right]
    \right].
  \end{equation}
  Since $X\in(0,1)$ almost surely, write
  \[
    u_X=\frac{V_X}{X}=q+\lambda\frac{1-X}{X},
    \qquad
    V_X=Xu_X,
    \qquad
    W_X=X(1-u_X).
  \]
  The two expectations in \eqref{eq:pge23-beta-first} therefore
  combine into the expectation of
  \[
    \begin{aligned}
      \frac {m}{n}\bigl(V_X^n+W_X^n\bigr)-XV_X^{n-1}
      &=X^n\left[
        \frac {m}{n}\bigl(u_X^n+(1-u_X)^n\bigr)-u_X^{n-1}
      \right]\\
      &=X^n\rho_{n,m}(u_X).
    \end{aligned}
  \]
  This proves \eqref{eq:pge23-beta-second}.

  We now derive \eqref{eq:pge23-beta-integral}.
  For $\lambda>0$, the probability density function of the Beta distribution gives
  \[
    \E\left[X^n\rho_{n,m}\left(q+\lambda\frac{1-X}{X}\right)\right]
    =\frac{\Gamma(2r)}{\Gamma(r)^2}
    \int_0^1 x^{n+r-1}(1-x)^{r-1}
    \rho_{n,m}\left(q+\lambda\frac{1-x}{x}\right)\dd x.
  \]
  We now make the substitution
  \[
    s=\lambda\frac{1-x}{x},
    \qquad
    x=\frac{\lambda}{\lambda+s},
    \qquad
    1-x=\frac{s}{\lambda+s},
    \qquad
    \dd x=-\frac{\lambda}{(\lambda+s)^2}\dd s.
  \]
  The endpoints $x=1$ and $x\to0^+$ correspond to $s=0$ and
  $s\to\infty$, respectively. Reversing the limits and using $n+2r=m$,
  we obtain the transformed weight
  \[
    \left(\frac{\lambda}{\lambda+s}\right)^{n+r-1}
    \left(\frac{s}{\lambda+s}\right)^{r-1}
    \frac{\lambda}{(\lambda+s)^2}
    =\frac{\lambda^{n+r}s^{r-1}}{(\lambda+s)^m}.
  \]
  Hence the expectation equals
  \[
    \frac{\Gamma(2r)}{\Gamma(r)^2}\lambda^{n+r}
    \int_0^\infty
    \frac{s^{r-1}}{(\lambda+s)^m}\rho_{n,m}(q+s)\dd s.
  \]
  Finally,
  \[
    \begin{aligned}
      \binom{m-1}{2r-1}\frac {n}{r}
      \frac{\Gamma(2r)}{\Gamma(r)^2}
      &=\frac{(m-1)!}{(2r-1)!\,n!}\cdot\frac {n}{r}
        \cdot\frac{(2r-1)!}{((r-1)!)^2}\\
      &=\frac{\Gamma(m)}{r\Gamma(n)\Gamma(r)^2}.
    \end{aligned}
  \]
  Substituting into \eqref{eq:pge23-beta-second} proves
  \eqref{eq:pge23-beta-integral}.
\end{proof}

\begin{lemma}\label{lem:pge23-ell-negative}
  Let $m\ge3$ be odd, let $r$ be an integer satisfying
  $1\le r\le(m-1)/2$, and set $n=m-2r$.
  If $0\le q\le1/2$ and $\lambda\le0$, then the diagonal polynomial
  defined in \eqref{eq:pge23-Ptilde} satisfies
  \[
    \widetilde P_{m,r}(q,\lambda)\ge0.
  \]
\end{lemma}

\begin{proof}
  From the definition \eqref{eq:pge23-rho} of $\rho_{n,m}$,
  \begin{align*}
    \rho_{n,m}(u)
    &=\frac {m}{n}\bigl(u^n+(1-u)^n\bigr)-u^{n-1} \\
    &=\frac {m}{n}(1-u)\bigl((1-u)^{n-1}-u^{n-1}\bigr)
    +\left(\frac {m}{n}-1\right)u^{n-1}.
  \end{align*}
  This identity follows by expanding
  $(1-u)((1-u)^{n-1}-u^{n-1})=(1-u)^n-u^{n-1}+u^n$.
  Since $n=m-2r>0$, we have $m/n-1=2r/n>0$.
  Also, as $n$ is odd, $n-1$ is even.
  Thus, in this decomposition of $\rho_{n,m}(u)$, the second term
  $\left(\frac{m}{n}-1\right)u^{n-1}$ is nonnegative for every
  $u\in\R$.
  The first term, $\frac{m}{n}(1-u)\bigl((1-u)^{n-1}-u^{n-1}\bigr)$,
  is nonnegative whenever $1-u\ge0$ and $|1-u|\ge|u|$, both of which hold
  for $u\le1/2$.

  Since $X\sim\BetaLaw(r,r)$ takes values in $(0,1)$, we have $(1-X)/X\ge0$, so $\lambda\le0$ gives
  \[
    q+\lambda\,\frac{1-X}{X}\le q\le\frac{1}{2}.
  \]
  Hence $X^n\rho_{n,m}\bigl(q+\lambda(1-X)/X\bigr)\ge0$, and \eqref{eq:pge23-beta-second} gives $\widetilde P_{m,r}(q,\lambda)\ge0$.
\end{proof}

\subsection{Gamma representation and the moment bound}
\label{subsec:gamma-positivity}

It now suffices to consider the case $\lambda>0$.
Unlike in the case $\lambda\le0$, the argument $q+s$ of $\rho_{n,m}$ in
\eqref{eq:pge23-beta-integral} ranges over $[q,\infty)$, where
$\rho_{n,m}$ can be negative; for example, $\rho_{5,7}(0.7)<0$.
The integrand can therefore no longer be handled pointwise. We must instead
prove that the integral as a whole is nonnegative.

Since the coefficient multiplying the integral in
\eqref{eq:pge23-beta-integral} is positive for $\lambda>0$, it remains to show that
\[
  \int_0^\infty
  \frac{s^{r-1}}{(\lambda+s)^m}\rho_{n,m}(q+s)\dd s\ge0.
\]
The difficulty is that $\lambda$ and $s$ are coupled through
$(\lambda+s)^{-m}$. The following identity separates
them:
\begin{equation}\label{eq:pge23-laplace-kernel}
  \frac{1}{(\lambda+s)^m}
  =\frac{1}{\Gamma(m)}\int_0^\infty t^{m-1}e^{-t(\lambda+s)}\dd t.
\end{equation}
This formula follows by the change of variable $w=t(\lambda+s)$ in the
defining integral $\Gamma(m)=\int_0^\infty w^{m-1}e^{-w}\dd w$.

\begin{lemma}\label{lem:pge23-laplace-gamma}
  Let $m\ge3$ be odd, let $r$ be an integer satisfying
  $1\le r\le(m-1)/2$, and set $n=m-2r$.
  Let $Y\sim\GammaLaw(r,1)$.
  For every $q\in\R$ and $\lambda>0$,
  \begin{equation}\label{eq:pge23-laplace-gamma}
    \widetilde P_{m,r}(q,\lambda)
    =\frac{\lambda^{n+r}}{r\Gamma(n)\Gamma(r)}
    \int_0^\infty
    t^{m-r-1}e^{-\lambda t}
    \E\left[\rho_{n,m}\left(q+\frac {Y}{t}\right)\right]\dd t.
  \end{equation}
  Equivalently, if $Z\sim\GammaLaw(m-r,1)$ is independent of $Y$, then
  \begin{equation}\label{eq:pge23-gamma-expectation}
    \widetilde P_{m,r}(q,\lambda)
    =\binom{m-r-1}{r}
      \E\!\left[\rho_{n,m}\!\left(q+\frac{\lambda}{Z}Y\right)\right].
  \end{equation}
\end{lemma}

\begin{proof}
  By inserting \eqref{eq:pge23-laplace-kernel} into
  \eqref{eq:pge23-beta-integral} and exchanging the two integrals, we obtain
  \begin{align*}
    \widetilde P_{m,r}(q,\lambda)
    &=\frac{\Gamma(m)}{r\Gamma(n)\Gamma(r)^2}\,\lambda^{n+r}
      \int_0^\infty s^{r-1}\rho_{n,m}(q+s)
      \left(\frac{1}{\Gamma(m)}\int_0^\infty t^{m-1}e^{-t(\lambda+s)}\dd t\right)
      \dd s\\
    &=\frac{\lambda^{n+r}}{r\Gamma(n)\Gamma(r)^2}
      \int_0^\infty t^{m-1}e^{-\lambda t}
      \left(\int_0^\infty s^{r-1}e^{-ts}\rho_{n,m}(q+s)\dd s\right)\dd t.
  \end{align*}
  We justify this exchange by verifying absolute integrability. Integrating
  the absolute value first in $t$ and using
  \eqref{eq:pge23-laplace-kernel} gives
  \[
    \int_0^\infty\!\!\int_0^\infty
    s^{r-1}t^{m-1}e^{-t(\lambda+s)}\bigl|\rho_{n,m}(q+s)\bigr|\dd t\dd s
    =\Gamma(m)\int_0^\infty \frac{s^{r-1}}{(\lambda+s)^m}
    \bigl|\rho_{n,m}(q+s)\bigr|\dd s.
  \]
  It remains to show that the integral on the right is finite.
  Because $n$ is odd, the terms of degree $n$ in $u^n+(1-u)^n$ cancel.
  The coefficient of $u^{n-1}$ in $\rho_{n,m}(u)$ is $m-1>0$, so
  $\rho_{n,m}(u)$ has degree $n-1$ in $u$.
  Since $n=m-2r$, the integrand is therefore
  $O(s^{r+n-2-m})=O(s^{-r-2})$ as $s\to\infty$, and hence is
  integrable at infinity. As $s\to0$,
  it is $O(s^{r-1})$, which is integrable because $r\ge1$.

  We now compute the inner integral in $s$ in the expression for
  $\widetilde P_{m,r}(q,\lambda)$ obtained after exchanging the two integrations.
  For fixed $t>0$, substituting $y=ts$, so that $s=y/t$ and
  $\dd s=\dd y/t$, gives
  \[
    \int_0^\infty s^{r-1}e^{-ts}\rho_{n,m}(q+s)\dd s
    =t^{-r}\int_0^\infty y^{r-1}e^{-y}
     \rho_{n,m}\left(q+\frac {y}{t}\right)\dd y
    =\Gamma(r)\,t^{-r}
     \E\left[\rho_{n,m}\left(q+\frac {Y}{t}\right)\right],
  \]
  where the last equality holds because $Y$ has density $y^{r-1}e^{-y}/\Gamma(r)$ on $(0,\infty)$.
  Substituting this for the inner integral obtained after the exchange gives \eqref{eq:pge23-laplace-gamma}.
  To obtain \eqref{eq:pge23-gamma-expectation}, substitute $z=\lambda t$
  in this integral. Since $m-r=n+r$, the powers of $\lambda$ cancel.
  The probability density function of $Z$ is
  $z^{m-r-1}e^{-z}/\Gamma(m-r)$, so the coefficient multiplying the
  expectation is
  \[
    \frac{\Gamma(m-r)}{r\Gamma(n)\Gamma(r)}
    =\binom{m-r-1}{r}.
  \]
  This proves the second identity in the lemma.
\end{proof}

In \eqref{eq:pge23-laplace-gamma}, the parameter $\lambda$ occurs only in
the positive coefficient $\lambda^{n+r}/(r\Gamma(n)\Gamma(r))$ of the integral
and the positive weight
$t^{m-r-1}e^{-\lambda t}$. The expectation
$\E\left[\rho_{n,m}(q+Y/t)\right]$ is independent of $\lambda$.
Thus, to prove $\widetilde P_{m,r}(q,\lambda)\ge0$ for every $\lambda>0$,
it suffices to prove that this expectation is nonnegative for every $t>0$.

We now seek to prove $\E\left[\rho_{n,m}(q+xY)\right]\ge0$ for $0\le q\le1/3$,
where $x=1/t>0$. The following lemma expands $\rho_{n,m}$ around $u=1/2$.

\begin{lemma}\label{lem:pge23-rho-centered}
  Let $n\ge1$ be odd, let $r\ge1$ be an integer, and set $m=n+2r$.
  For $0\le j\le(n-1)/2$, let
  \[
    a_j=2^{2j+1-n}\binom n{2j}>0.
  \]
  Then, for every $v\in\R$,
  \begin{equation}\label{eq:pge23-rho-centered}
    n\rho_{n,m}\left(\frac12+v\right)
    =2ra_0+\sum_{j=1}^{(n-1)/2}
    a_j\bigl(2(r+j)v^{2j}-jv^{2j-1}\bigr).
  \end{equation}
\end{lemma}

\begin{proof}
  Let $A_n(u)=u^n+(1-u)^n$. Since
  \[
    (1-u)A_n'(u)=nu^{n-1}-nA_n(u),
  \]
  the definition \eqref{eq:pge23-rho} and $m-n=2r$ give
  \begin{equation}\label{eq:pge23-rho-derivative}
    n\rho_{n,m}(u)=2rA_n(u)-(1-u)A_n'(u).
  \end{equation}
  By the binomial theorem,
  \begin{align*}
    A_n\left(\frac12+v\right)
    &=\left(\frac12+v\right)^n+\left(\frac12-v\right)^n\\
    &=2\sum_{2j\le n}\binom n{2j}\left(\frac12\right)^{n-2j}v^{2j}
      =\sum_{j=0}^{(n-1)/2}a_jv^{2j}.
  \end{align*}
  Differentiating in $v$ and substituting into
  \eqref{eq:pge23-rho-derivative} gives
  \[
    n\rho_{n,m}\left(\frac12+v\right)
    =2r\sum_{j=0}^{(n-1)/2}a_jv^{2j}
    -\left(\frac12-v\right)
      \sum_{j=1}^{(n-1)/2}2ja_jv^{2j-1},
  \]
  which simplifies to \eqref{eq:pge23-rho-centered}.
\end{proof}

Since all $a_j$ are positive, \eqref{eq:pge23-rho-centered} shows that it
suffices to prove
\[
  j\E\left[v^{2j-1}\right]\le2(r+j)\E\left[v^{2j}\right],
  \qquad 1\le j\le(n-1)/2,
\]
where $v=q+xY-1/2$. The following Gamma moment bound will establish these
inequalities.

\begin{lemma}[Odd-to-even Gamma moment bound]\label{lem:pge23-gamma-moment}
  Let $r,j\ge1$ be integers and $Y\sim\GammaLaw(r,1)$.
  For every $z\ge0$,
  \begin{equation*}
    3j\,\E\left[(zY-1)^{2j-1}\right]
    \le(r+j)\,\E\left[(zY-1)^{2j}\right].
  \end{equation*}
\end{lemma}

For $z>0$, this is the moment bound \eqref{eq:intro-moment-bound}
after rescaling with $b=1/z$.
The proof is given in \Cref{app:gamma-moment}.

The next proposition uses \Cref{lem:pge23-gamma-moment} with the same
Gamma shape parameter $r$ and, for each summand in
\eqref{eq:pge23-rho-centered}, its index $j$. For $0\le q\le1/3$
and $x\ge0$, we set
\[
  \delta=\frac12-q\ge\frac16,
  \qquad z=\frac{x}{\delta}\ge0,
  \qquad v=q+xY-\frac12=\delta(zY-1).
\]
Thus, the lemma bounds the moments of $v/\delta=zY-1$.
Using $\E\left[v^k\right]=\delta^k\E\left[(zY-1)^k\right]$, the following proposition proves that
\[
  \E\left[2(r+j)v^{2j}-jv^{2j-1}\right]\ge0,
  \qquad 1\le j\le(n-1)/2.
\]
Together with \eqref{eq:pge23-rho-centered}, this gives
the nonnegativity of
$\E\left[\rho_{n,m}(q+xY)\right]$ needed in
\eqref{eq:pge23-laplace-gamma}.

\begin{proposition}\label{prop:pge23-gamma-positive}
  Let $n\ge1$ be odd, let $r\ge1$ be an integer, and set $m=n+2r$.
  If $0\le q\le1/3$, then for every $x\ge0$,
  \begin{equation}\label{eq:pge23-shifted-positive}
    \E\left[\rho_{n,m}(q+xY)\right]\ge0,
    \qquad Y\sim\GammaLaw(r,1).
  \end{equation}
\end{proposition}

\begin{proof}
  Let
  \[
    \delta=\frac12-q\ge\frac16,
    \qquad z=\frac{x}{\delta}\ge0,
    \qquad v=q+xY-\frac12=\delta(zY-1).
  \]
  After taking expectations in \eqref{eq:pge23-rho-centered}, it suffices to
  bound each summand.
  For $j\ge1$,
  \begin{align*}
    &\E\left[
    2(r+j)\bigl(\delta(zY-1)\bigr)^{2j}
    -j\bigl(\delta(zY-1)\bigr)^{2j-1}
    \right]\\
    &\qquad=2(r+j)\delta^{2j}\E\left[(zY-1)^{2j}\right]
      -j\delta^{2j-1}\E\left[(zY-1)^{2j-1}\right]\\
    &\qquad\ge2(r+j)\delta^{2j}\E\left[(zY-1)^{2j}\right]
      -\frac{r+j}{3}\delta^{2j-1}\E\left[(zY-1)^{2j}\right]\\
    &\qquad=\delta^{2j}(r+j)
    \left(2-\frac{1}{3\delta}\right)
    \E\left[(zY-1)^{2j}\right]
    \ge0,
  \end{align*}
  where the inequality follows from \Cref{lem:pge23-gamma-moment}. The last
  step uses $2\delta\ge1/3$ and $\E\left[(zY-1)^{2j}\right]\ge0$.
  Since every $a_j$ is positive, we conclude
  \[
    n\,\E\left[\rho_{n,m}(q+xY)\right]\ge2ra_0>0,
  \]
  which proves \eqref{eq:pge23-shifted-positive}.
\end{proof}

\begin{remark}\label{rem:pge23-threshold}
  For $p>1/2$, the condition $2-1/(3\delta)\ge0$, where
  $\delta=p-1/2$, is equivalent to $p\ge2/3$.
  This is the only step in the proof of \Cref{prop:pge23-gamma-positive}
  requiring $p\ge2/3$; for $1/2<p<2/3$, the lower bound
  \[
    \delta^{2j}(r+j)\left(2-\frac{1}{3\delta}\right)
    \E\left[(zY-1)^{2j}\right]
  \]
  for $\E\left[2(r+j)v^{2j}-jv^{2j-1}\right]$ is negative, so it does
  not establish the required nonnegativity.
\end{remark}

\medskip

\begin{proof}[Proof of the inequalities in \Cref{thm:main,thm:commonality}]
  First suppose that $W$ is a step graphon on a partition into intervals of equal length, and write $U=1-W$ and $q=1-p$.
  When $0\le q\le1/3$, \Cref{lem:pge23-dirichlet} reduces the nonnegativity of $P_{m,r}$ to
  $\widetilde P_{m,r}(q,\lambda)\ge0$ for $-1/2\le\lambda\le1/2$.
  This follows from \Cref{lem:pge23-ell-negative} when $\lambda\le0$.
  When $\lambda>0$, \Cref{lem:pge23-laplace-gamma} expresses
  $\widetilde P_{m,r}(q,\lambda)$ as an integral of expectations of the form
  $\E\left[\rho_{n,m}(q+xY)\right]$, which are nonnegative by
  \Cref{prop:pge23-gamma-positive}.
  Thus, $P_{m,r}$ is nonnegative on $[-1/2,1/2]^r$.
  For $m\in\{3,5,7\}$, \Cref{prop:short-cycles} gives the nonnegativity of $P_{m,r}$ throughout the larger range $0\le q\le1/2$.

  In both cases, the nonnegativity of $\mu$ gives $\Phi_m\ge0$, so \Cref{lem:polynomial-reduction} proves \eqref{eq:intro-commonality} for step graphons in the corresponding density ranges.
  The approximation argument used in \Cref{lem:step-reduction}, which preserves edge density, extends this inequality to arbitrary graphons.
  Finally, \eqref{eq:cycle-path-comparison} gives
  \[
    \tC mW\ge p^m+q^m-q^{m-1}=p^m-pq^{m-1},
  \]
  which is \eqref{eq:intro-main}.
\end{proof}

\section{Equality and qualitative stability}
\label{sec:equality-stability}

This section characterises equality in the odd-cycle Goodman bound \eqref{eq:intro-main}
by analysing the two inequalities used to derive it: the strengthened
extended commonality inequality \eqref{eq:intro-commonality} and the
comparison between cycle and path densities \eqref{eq:cycle-path-comparison}.
We first prove in \Cref{prop:regularity-gap} that equality in
\eqref{eq:intro-commonality} forces $W$ to have constant degree $p$.
We then show in \Cref{prop:main-equality} that, under this regularity condition,
equality in \eqref{eq:cycle-path-comparison} forces the complement $1-W$
to be a union of equal cliques.
Together, these conditions identify the balanced complete multipartite graphons as the only graphons attaining the bound when $2/3\le p<1$.
Finally, we use this uniqueness and graphon compactness to prove qualitative
stability at densities $p=1-1/k$, $k\ge3$, in \Cref{cor:stability}.

For a graphon $W$ of edge density $p$, define its degree function by
\[
  \dW(x)=\int_0^1 W(x,y)\dd y.
\]
We call $W$ \emph{$p$-regular} if $\dW(x)=p$ for almost every $x$.
With $U=1-W$ and $q=1-p$, define $\Psi_m(W)$ for any graphon $W$
as the difference between the left- and right-hand sides of
\eqref{eq:intro-commonality}:
\begin{equation}\label{eq:equality-gap}
  \Psi_m(W):=\tC mW+\tC mU-\tP{m-1}U-\bigl(p^m-pq^{m-1}\bigr).
\end{equation}
The gap in the odd-cycle bound splits as
\begin{equation}\label{eq:one-colour-gap}
  \tC mW-\bigl(p^m-pq^{m-1}\bigr)
  =\Psi_m(W)+\bigl[\tP{m-1}U-\tC mU\bigr].
\end{equation}
For $p\ge2/3$, both terms on the right are nonnegative: the first by
\eqref{eq:intro-commonality}, proved in \Cref{sec:excursion-positivity},
and the second by \eqref{eq:cycle-path-comparison}.
Thus, equality in \eqref{eq:intro-main} holds exactly when both terms vanish.
We first characterise the condition $\Psi_m(W)=0$ in
\Cref{prop:regularity-gap}, and then analyse equality in the cycle-path
comparison in \Cref{prop:main-equality}.

\subsection{Degree variance and regularity}
\label{subsec:degree-variance}

For a step graphon on a partition into intervals of equal length,
\Cref{lem:polynomial-reduction} gives
$\Psi_m(W)=\Phi_m$, where $\Phi_m$ is the sum of spectral integrals defined 
in \eqref{eq:Phi-expansion}.
Keeping the term $r=(m-1)/2$ in \eqref{eq:Phi-expansion} gives a degree-variance bound, which we extend to arbitrary graphons by approximation.

\begin{proposition}
\label{prop:regularity-gap}
  Let $m\ge3$ be odd, and let $W$ be a graphon of edge density $p\ge2/3$.
  Then,
  \begin{equation}\label{eq:degree-variance-gap}
    \Psi_m(W)\ge(m-1)
    \left(\int_0^1(\dW(x)-p)^2\dd x\right)^{(m-1)/2}.
  \end{equation}
  Moreover, $\Psi_m(W)=0$ if and only if $W$ is $p$-regular.
  Consequently, equality holds in \eqref{eq:intro-commonality} of
  \Cref{thm:commonality} if and only if $W$ is $p$-regular.
\end{proposition}

\begin{proof}
  First suppose that $W$ is constant on each rectangle $I_a\times I_b$,
  where $I_1,\ldots,I_N$ partition $[0,1]$ into intervals of length $1/N$,
  as in \Cref{subsec:step-matrices}.
  \Cref{lem:polynomial-reduction} and the definition \eqref{eq:Phi-expansion} of $\Phi_m$ give
  \[
    \Psi_m(W)=\Phi_m
    =\sum_{r=1}^{(m-1)/2}
      \int\!\cdots\!\int
      P_{m,r}(q;\lambda_1,\ldots,\lambda_r)
      \prod_{i=1}^r\dd\mu(\lambda_i).
  \]
  The nonnegativity of $P_{m,r}$ in \eqref{eq:pge23-positivity-goal}, proved in
  \Cref{sec:excursion-positivity}, makes every integrand in
  this sum nonnegative on the support of $\mu$. Since $\mu$ is nonnegative,
  we obtain a lower bound by retaining only the summand with $r=(m-1)/2$.
  For this summand, $n=m-2r=1$, so \eqref{eq:Pmr} gives
  \begin{align*}
    P_{m,r}(q;\lambda_1,\ldots,\lambda_r)
    ={}&\frac mr\Bigl[
      h_1(\rep pr,-\lambda_1,\ldots,-\lambda_r)
      +h_1(\rep qr,\lambda_1,\ldots,\lambda_r)
    \Bigr]\\
    &-h_0(\rep q{(r+1)},\lambda_1,\ldots,\lambda_r)\\
    ={}&\frac mr\Bigl[
      \Bigl(rp-\sum_{i=1}^r\lambda_i\Bigr)
      +\Bigl(rq+\sum_{i=1}^r\lambda_i\Bigr)
    \Bigr]-1\\
    ={}&m-1.
  \end{align*}
  Here we used that $h_1$ is the sum of its variables, $h_0=1$, and $p+q=1$.
  Integrating the constant $m-1$ against the $r$-fold product of $\mu$ gives
  \[
    \Psi_m(W)\ge(m-1)\mu(\R)^{(m-1)/2}.
  \]
  To obtain \eqref{eq:degree-variance-gap} from this bound, we use the identity
  \[
    \mu(\R)=\norm g^2
    =\int_0^1(\dW(x)-p)^2\dd x.
  \]
  The first equality is given by \eqref{eq:moments}.
  To verify the second, let $d_a$ be the constant value of $\dW$
  on the partition interval $I_a$, so
  $d_a=\frac1N\sum_{b=1}^N W_{ab}$.
  The definitions $g=M_Ue-qe$, $e_a=1/\sqrt N$, and
  $(M_U)_{ab}=U_{ab}/N$ give
  \[
    g_a=\frac1{\sqrt N}\left(\frac1N\sum_{b=1}^N U_{ab}-q\right)
    =\frac{1-d_a-q}{\sqrt N}
    =\frac{p-d_a}{\sqrt N},
  \]
  where we used $U=1-W$ and $q=1-p$.
  Since each $I_a$ has length $1/N$, we obtain
  \[
    \norm g^2=\frac1N\sum_{a=1}^N(p-d_a)^2
    =\int_0^1(\dW(x)-p)^2\dd x,
  \]
  proving \eqref{eq:degree-variance-gap} for step graphons of this form.

  For an arbitrary $W$, let $W_s$ be the step-graphon approximations on
  partitions into intervals of equal length constructed in the proof of
  \Cref{lem:step-reduction}. Each $W_s$ has the same edge density $p$ as
  $W$, so the bound just proved gives
  \[
    \Psi_m(W_s)\ge(m-1)
    \left(\int_0^1(d_{W_s}(x)-p)^2\dd x\right)^{(m-1)/2}.
  \]
  All homomorphism densities of $W_s$ and $1-W_s$ converge to those of
  $W$ and $1-W$, respectively. By the definition \eqref{eq:equality-gap},
  the left-hand side therefore converges to $\Psi_m(W)$.
  For the right-hand side, express the degree variance through the
  two-edge path density:
  \begin{align*}
    \int_0^1(d_{W_s}(x)-p)^2\dd x
    =\Big(\tP2{W_s}-p^2\Big)
    \longrightarrow\Big(\tP2W-p^2\Big)
    =\int_0^1(\dW(x)-p)^2\dd x.
  \end{align*}
  Thus, the right-hand side also converges to the corresponding expression
  for $W$. Taking limits in the inequality proves
  \eqref{eq:degree-variance-gap} for arbitrary graphons.
  Moreover, if $\Psi_m(W)=0$, \eqref{eq:degree-variance-gap} forces $W$ to be $p$-regular.

  Conversely, suppose that $W$ is $p$-regular.
  The degree function of $U=1-W$ is then the constant $q$, so
  \[
    \Tker W\one=\dW=p\one,
    \qquad
    \Tker U\one=q\one .
  \]
  Both operators are self-adjoint, so each leaves the orthogonal complement
  $\one^\perp=\{f\in L^2[0,1]:\int_0^1f(x)\dd x=0\}$ invariant.
  On $\one^\perp$, the two operators are negatives of each other: the identity
  $W+U\equiv1$ gives
  \[
    \Tker Wf+\Tker Uf=\Bigl(\int_0^1f(y)\dd y\Bigr)\one=\ip f\one\one
    \qquad\text{for all }f\in L^2[0,1],
  \]
  which is the orthogonal projection onto the constants and therefore vanishes on
  $\one^\perp$.
  Write $\nu_1,\nu_2,\ldots$ for the eigenvalues of the restriction of $\Tker W$ to
  $\one^\perp$.
  The eigenvalues of $\Tker W$ are then $p$ together with the $\nu_j$, and those of
  $\Tker U$ are $q$ together with the $-\nu_j$.
  Since $m$ is odd, the terms $\nu_j^m$ and $(-\nu_j)^m$ cancel in pairs, so the
  trace formula gives
  \[
    \tC mW+\tC mU
    =\Bigl(p^m+\sum_j\nu_j^m\Bigr)+\Bigl(q^m+\sum_j(-\nu_j)^m\Bigr)
    =p^m+q^m.
  \]
  Iterating $\Tker U\one=q\one$ gives $\Tker U^{m-1}\one=q^{m-1}\one$, whence
  \[
    \tP{m-1}U=\ip\one{\Tker U^{m-1}\one}=q^{m-1}.
  \]
  Substituting both evaluations in \eqref{eq:equality-gap} and using $1-p=q$,
  \[
    \Psi_m(W)
    =\bigl(p^m+q^m\bigr)-q^{m-1}-\bigl(p^m-pq^{m-1}\bigr)
    =q^m-(1-p)q^{m-1}
    =0. \qedhere
  \]
\end{proof}

\subsection{Equality in the odd-cycle bound}
\label{subsec:main-equality}

By \eqref{eq:one-colour-gap} and \Cref{prop:regularity-gap}, equality in
\eqref{eq:intro-main} is equivalent to $W$ being $p$-regular and
$\tC mU=\tP{m-1}U$.
For $p<1$, the next proposition shows that these conditions force $U$ to
be a union of equal cliques, and hence $W$ to be a balanced complete
multipartite graphon.

\begin{proposition}
\label{prop:main-equality}
  Let $m\ge3$ be odd, and let $W$ be a graphon of edge density $2/3\le p<1$.
  Equality holds in \Cref{thm:main} if and only if $p=1-1/k$ for an integer $k\ge3$ and, up to a measure-preserving relabelling and null sets, $W=W_k$.
\end{proposition}

\begin{proof}
  The backward implication follows from \eqref{eq:intro-kpartite-value}.
  For the forward implication, suppose that equality holds in
  \eqref{eq:intro-main}, and write $U=1-W$ and $q=1-p>0$.
  Then, \eqref{eq:one-colour-gap} and \Cref{prop:regularity-gap} imply
  that $W$ is $p$-regular and $\tC mU=\tP{m-1}U$.
  Since $d_U(x)=1-\dW(x)=q$ almost everywhere, $U$ is $q$-regular
  and $\Tker U\one=q\one$. Consequently,
  \[
    \tC mU=\tP{m-1}U
    =\ip\one{\Tker U^{m-1}\one}=q^{m-1}.
  \]

  We use this cycle density to determine the possible eigenvalues of
  $\Tker U$. We first show that its operator norm is at most $q$.
  Since $U\ge0$, we may split $U=U^{1/2}\cdot U^{1/2}$.
  For $f\in L^2[0,1]$, applying the Cauchy--Schwarz inequality in the
  variable $y$ gives
  \[
    \bigl|\Tker Uf(x)\bigr|^2
    =\left|\int_0^1 U(x,y)^{1/2}\cdot U(x,y)^{1/2}f(y)\dd y\right|^2
    \le\left(\int_0^1 U(x,y)\dd y\right)
      \left(\int_0^1 U(x,y)\bigl|f(y)\bigr|^2\dd y\right).
  \]
  The first factor is the degree $d_U(x)$. Since $U$ is $q$-regular,
  $d_U(x)=q$ for almost every $x$, giving
  \[
    \bigl|\Tker Uf(x)\bigr|^2\le q\int_0^1 U(x,y)\bigl|f(y)\bigr|^2\dd y .
  \]
  Integrating in $x$, exchanging the order of integration, and using $q$-regularity a
  second time, now in the variable $y$,
  \[
    \norm{\Tker Uf}_2^2
    \le q\int_0^1\bigl|f(y)\bigr|^2
      \left(\int_0^1 U(x,y)\dd x\right)\dd y
    =q^2\norm f_2^2 .
  \]
  This proves the operator-norm bound. Since $\Tker U$ is self-adjoint,
  its eigenvalues $\lambda_i$ lie in $[-q,q]$.
  Since $|\lambda_i|\le q$ and $m\ge3$, each eigenvalue satisfies
  \[
    \lambda_i^m\le|\lambda_i|^m
    =\lambda_i^2|\lambda_i|^{m-2}
    \le q^{m-2}\lambda_i^2.
  \]
  Since $U$ is square-integrable, its eigenvalues, counted with
  multiplicity, satisfy
  \[
    \sum_i\lambda_i^2=\intsq U(x,y)^2\ddxy.
  \]
  The trace formula and the pointwise inequality $U(x,y)^2\le U(x,y)$ therefore give
  \begin{equation}\label{eq:regular-complement-cycle}
    q^{m-1}=\sum_i\lambda_i^m
    \le q^{m-2}\sum_i\lambda_i^2
    =q^{m-2}\intsq U(x,y)^2\ddxy
    \le q^{m-1}.
  \end{equation}
  The two ends of \eqref{eq:regular-complement-cycle} agree, so equality
  holds in both inequalities. Equality in the last inequality forces
  $U(x,y)^2=U(x,y)$ for almost every $(x,y)$, so $U$ takes values in $\{0,1\}$.
  Equality in the eigenvalue inequality forces $|\lambda_i|=q$ for every
  nonzero eigenvalue. Since $m$ is odd, a negative eigenvalue would satisfy
  $\lambda_i^m<|\lambda_i|^m$, so every nonzero $\lambda_i$ equals $q$.
  The spectral theorem now gives
  \[
    \Tker U=qP,
    \qquad P=\text{the orthogonal projection onto }\ker(\Tker U-q).
  \]
  Since $P^2=P$, this gives $\Tker U^2=q^2P^2=q^2P=q\Tker U$.
  Square-integrable kernels defining the same integral operator agree
  almost everywhere. Comparing the kernels of $\Tker U^2$ and
  $q\Tker U$, and using the symmetry of $U$, therefore gives
  \begin{equation}\label{eq:equal-neighbourhoods}
    \int_0^1 U(x,z)U(y,z)\dd z=qU(x,y)
    \quad\text{for almost every }(x,y).
  \end{equation}

  Write $N(x)=\{y:U(x,y)=1\}$ and $|A|$ for the Lebesgue measure of $A\subseteq[0,1]$, so that $|N(x)|=q$ for almost every $x$.
  The left side of \eqref{eq:equal-neighbourhoods} is $|N(x)\cap N(y)|$ and $U$ is $\{0,1\}$-valued, so for almost every pair $(x,y)$,
  \[
    \bigl|N(x)\cap N(y)\bigr|
    =\begin{cases}
      q&\text{if }U(x,y)=1,\\
      0&\text{if }U(x,y)=0.
    \end{cases}
  \]
  For almost every pair $(x,y)$ with $U(x,y)=1$, both neighbourhoods
  and their intersection have measure $q$. Hence
  \[
    |N(x)\setminus N(y)|=|N(y)\setminus N(x)|
    =q-|N(x)\cap N(y)|=0.
  \]
  Thus, the two neighbourhoods differ only on a set of measure zero.
  For almost every pair with $U(x,y)=0$, their intersection has measure
  zero, so the neighbourhoods are disjoint apart from a set of measure zero.
  In the construction below, equality, containment, and disjointness of sets
  are understood up to null sets. Statements that $U=1$ or $U=0$ on a region mean that the
  asserted value holds for almost every pair of points in that region.

  We now partition $[0,1]$ into disjoint sets $B_1,\ldots,B_k$ of
  measure $q$, with $U=1$ within each set and $U=0$ between different
  sets. At stage $j$, we have chosen pairwise disjoint measurable sets
  $B_1,\ldots,B_j$, and the remaining set is
  \[
    R_j=[0,1]\setminus\bigcup_{i=1}^j B_i.
  \]
  We maintain the following invariant:
  \begin{enumerate}[label=(\roman*)]
    \item For each $i\le j$, $|B_i|=q$ and $N(y)=B_i$ for almost
      every $y\in B_i$. Thus, $B_i$ is a clique with no edges to its
      complement.
    \item For almost every $x\in R_j$, $|N(x)|=q$ and
      $|N(x)\setminus R_j|=0$. Thus, almost every point in the
      remainder still has all its neighbours in the remainder.
  \end{enumerate}
  Initially, $j=0$ and $R_0=[0,1]$. Condition (i) is vacuous, and
  condition (ii) follows from the $q$-regularity of $U$.

  Suppose the invariant holds at stage $j$ and $|R_j|>0$.
  By Fubini's theorem and the neighbourhood relation above, for almost
  every $x\in[0,1]$ we have
  \[
    U(x,y)=1\quad\Longrightarrow\quad N(y)=N(x)
    \qquad\text{for almost every }y\in[0,1].
  \]
  Since $R_j$ has positive measure, choose such an $x\in R_j$ also
  satisfying condition (ii): $|N(x)|=q$ and
  $|N(x)\setminus R_j|=0$. Set
  \[
    B_{j+1}=N(x)\cap R_j,
    \qquad R_{j+1}=R_j\setminus B_{j+1}.
  \]
  The new block is disjoint from the earlier blocks and has measure
  $q$. For almost every $y\in B_{j+1}$, we have
  $U(x,y)=1$, so the neighbourhood relation gives
  $N(y)=N(x)=B_{j+1}$. This verifies condition (i) for the new block;
  the earlier blocks are unchanged.

  To verify condition (ii), we must show that almost every point of
  $R_{j+1}$ still has a neighbourhood of measure $q$ contained in
  $R_{j+1}$. For almost every $z\in B_{j+1}$, the identity
  $N(z)=B_{j+1}$ shows that $z$ has no neighbours in $R_{j+1}$.
  Symmetry of $U$ and Fubini's theorem therefore give
  \[
    \begin{aligned}
      0&=\int_{B_{j+1}}\int_{R_{j+1}}U(z,y)\dd y\dd z\\
       &=\int_{R_{j+1}}\int_{B_{j+1}}U(y,z)\dd z\dd y
        =\int_{R_{j+1}}|N(y)\cap B_{j+1}|\dd y.
    \end{aligned}
  \]
  The last integrand is nonnegative, so $|N(y)\cap B_{j+1}|=0$
  for almost every $y\in R_{j+1}$. Together with condition (ii) at
  stage $j$, this gives
  \[
    |N(y)\setminus R_{j+1}|=|N(y)\setminus R_j|=0,
    \qquad |N(y)|=q
  \]
  for almost every $y\in R_{j+1}$. Thus, condition (ii) is preserved.

  Each step removes measure $q>0$, so $|R_j|=1-jq$ and at most
  $\lfloor1/q\rfloor$ steps are possible. The construction can continue
  whenever the remainder has positive measure, so it must end with
  $|R_k|=0$ for some $k$. Hence $B_1,\ldots,B_k$ cover $[0,1]$,
  and $kq=1$. By condition (i), $U$ consists of the
  $k$ disjoint cliques $B_i\times B_i$, so its complement is $W_k$
  after relabelling.
  Finally, $p=1-1/k\ge2/3$ implies $k\ge3$.
\end{proof}

At $p=1$, the only graphon is $W=1$ almost everywhere, which attains equality
in the odd-cycle Goodman bound \eqref{eq:intro-main} of \Cref{thm:main}.
This completes the equality assertion in \Cref{thm:main}.

\subsection{Qualitative stability}
\label{subsec:qualitative-stability}

The uniqueness just proved, together with compactness and continuity in cut distance (\Cref{subsec:graphon-preliminaries}), gives the following corollary.

\begin{corollary}[Qualitative stability]\label{cor:stability}
  Fix an odd integer $m\ge3$.
  For every integer $k\ge3$ and every $\varepsilon>0$, there exists $\delta>0$ such that every graphon $V$ satisfying
  \[
    \bigl|\edens V-(1-1/k)\bigr|\le\delta,
    \qquad
    \tC mV\le\tC m{W_k}+\delta
  \]
  satisfies $\delta_\square(V,W_k)<\varepsilon$.
\end{corollary}

\begin{proof}
  Fix an odd integer $m\ge3$ and an integer $k\ge3$, and set $p_k=1-1/k$.
  Suppose the conclusion fails for some $\varepsilon>0$.
  Taking $\delta=1/j$ for $j\in\mathbb N$ produces graphons $W^{(j)}$ with
  \[
    \bigl|\edens{W^{(j)}}-p_k\bigr|\le\frac1j,
    \qquad
    \tC m{W^{(j)}}\le p_k^m-p_k(1-p_k)^{m-1}+\frac1j,
    \qquad
    \delta_\square(W^{(j)},W_k)\ge\varepsilon .
  \]
  By compactness in cut distance (\Cref{subsec:graphon-preliminaries}), after passing to a subsequence, we may assume that $W^{(j)}$ converges in cut distance to a graphon $W$.
  Edge and cycle densities are continuous in cut distance, so the first two conditions pass to the limit as
  \[
    \edens W=p_k,
    \qquad
    \tC mW\le p_k^m-p_k(1-p_k)^{m-1}.
  \]
  Since $k\ge3$ gives $p_k\ge2/3$, \Cref{thm:main} applies to $W$ and supplies the reverse inequality, so $\tC mW=p_k^m-p_k(1-p_k)^{m-1}$; that is, equality holds in \Cref{thm:main}.
  \Cref{prop:main-equality} then makes $W$ equivalent to $W_k$.
  The triangle inequality gives
  \[
    \begin{aligned}
      \delta_\square(W^{(j)},W_k)
      &\le\delta_\square(W^{(j)},W)+\delta_\square(W,W_k)\\
      &=\delta_\square(W^{(j)},W)\longrightarrow0.
    \end{aligned}
  \]
  This contradicts $\delta_\square(W^{(j)},W_k)\ge\varepsilon>0$
  for every $j$.
\end{proof}

\section{Concluding remarks and further questions}
\label{sec:conclusion}

We have extended Goodman's bound to every odd cycle for graphons of edge
density $p\ge2/3$ (\Cref{thm:main}).
At each density $p=1-1/k$ with $k\ge3$, the balanced complete
$k$-partite graphon $W_k$ is the unique minimiser up to graphon equivalence
and satisfies qualitative stability in cut distance (\Cref{cor:stability}).

The restriction $p\ge2/3$ enters when we apply the odd-to-even Gamma
moment bound to prove the nonnegativity of
$\E\rho_{n,m}(q+xY)$ in \Cref{prop:pge23-gamma-positive};
\Cref{rem:pge23-threshold} identifies the coefficient that requires this
restriction. The strengthened extended commonality inequality can fail
when $1/2<p<2/3$, so the general argument does not extend directly to
that interval. For example, let $W$ be constant on the products of two
parts of measure $1/2$, with value matrix
\[
  \frac1{250}\begin{pmatrix}50&193\\193&65\end{pmatrix}.
\]
Then, $p=501/1000$, and the matrix formulas in
\Cref{subsec:step-matrices} give
\[
  \Psi_9(W)=-\frac{227031617768589}{5\cdot10^{21}}<0.
\]
Thus, this graphon violates \eqref{eq:intro-commonality} for $m=9$.
For $C_3$, $C_5$, and $C_7$, the direct computations in
\Cref{prop:short-cycles} prove both main inequalities throughout $p\ge1/2$.
The companion paper proves the odd-cycle Goodman bound
\eqref{eq:intro-main} for all odd cycle lengths on $1/2<p<2/3$ by different
methods.

A natural next question is quantitative stability. The degree-variance
estimate \eqref{eq:degree-variance-gap} already controls the variation in
degrees through the gap $\Psi_m(W)$ in
\eqref{eq:intro-commonality}. The proof of \Cref{cor:stability} gives no
explicit rate for the cut distance. At a fixed edge density $p=1-1/k$,
$k\ge3$, one can ask for an explicit bound on
$\delta_\square(W,W_k)$ in terms of the cycle-density excess
$\tC mW-\tC m{W_k}$.

The balanced complete multipartite graphons also suggest polynomial
bounds for other nonbipartite graphs. For such a graph $H$, let
$\chi_H$ be its chromatic polynomial, let
$\chi(H)$ be its chromatic number, and let $v(H)$ be its number of vertices.
Define
\[
  G_H(p):=\chi_H\left(\frac{1}{1-p}\right)(1-p)^{v(H)}.
\]
Since $\chi_H$ has degree $v(H)$, substituting $1/(1-p)$ produces terms
whose denominators are $(1-p)^j$ with $0\le j\le v(H)$.
The factor $(1-p)^{v(H)}$ therefore cancels every denominator, leaving
a polynomial in $p$. At $p=1$, we evaluate this polynomial.
At $p=1-1/k$, the number of proper $k$-colourings of $H$ gives
\[
  G_H(1-1/k)=\frac{\chi_H(k)}{k^{v(H)}}=\dens H{W_k}.
\]
Thus, $G_H$ is the unique polynomial in $p$ agreeing with these
multipartite densities for every integer $k\ge\chi(H)-1$.
For odd $m$, this construction recovers the polynomial studied here:
\[
  G_{C_m}(p)=p^m-p(1-p)^{m-1}.
\]
We say that $H$ satisfies its \emph{Goodman-type bound} if every graphon
$W$ with $\edens W=p\ge1-1/(\chi(H)-1)$ satisfies
\[
  \dens HW\ge G_H(p).
\]
For such a graph $H$, every $W_k$ with $k\ge\chi(H)-1$ attains the
bound and hence minimises the $H$-density at its edge density.

This leads to a broader question: which nonbipartite graphs satisfy their
Goodman-type bounds? The following examples and obstruction give some
directions for this classification; their proofs lie beyond the scope of
this paper.
\begin{itemize}
  \item A nonbipartite chordal graph whose maximal cliques all have the
  same size satisfies its Goodman-type bound.
  \item For integers $r\ge1$ and $k\ge1$, the complete join
  $K_r\vee C_{2k+1}$, i.e., the graph obtained by taking the disjoint union of $K_r$ and $C_{2k+1}$ and adding all edges between them, satisfies its Goodman-type bound.  Taking $r=1$ gives the odd wheels.
  \item A nonbipartite graph of even girth does not satisfy its
  Goodman-type bound. 
\end{itemize}
Characterising the graphs that satisfy their Goodman-type bounds would
explain when the values attained by balanced complete multipartite graphons
yield a polynomial lower bound valid also at the densities between those
points of equality.

\paragraph{AI use disclosure.}
The problem studied here arose from the neural graphon representation
framework of Kim, Baek, Lee, and Yang~\cite{kim2026implicit}.
The authors guided the initial phase of proof development, which lasted
about a month and involved preliminary computations and interactive proof
search with GPT-5.5, GPT-5.5 Pro, and Claude Fable 5.0.
Over the next three months, the authors worked with these models,
GPT-6.0 Astra, and Claude Opus 5.5 to simplify the proof and improve its exposition.
These efforts led to the development of the excursion-based proof,
a revision of the argument following \Cref{lem:pge23-ell-negative} using
\Cref{lem:pge23-gamma-moment}, and the equality and stability results.
The authors independently verified the written proof.

The Lean formalisation of the proof was carried out by Claude Fable 5.1
and Claude Opus 4.8, with only minimal input from the authors.
The authors checked the Lean formalisation and confirmed that its definitions
and theorem statements matched those in the paper.
GPT-5.6 Sol, GPT-6.0 Astra, Claude Opus 4.8, and Claude Opus 5.5 also assisted with drafting and revising
the manuscript.

\paragraph{Lean formalisation.}
The proofs of \Cref{thm:main,thm:commonality,prop:main-equality} are formalised in Lean 4, while the stability assertion in \Cref{cor:stability} is not formalised.
The formalisation is available at \url{https://github.com/mekty2012/OddCycleGoodman-lean}, and this repository also contains a proof of \Cref{thm:main} for $1/2<p<2/3$, which is not described in this paper.

There are some deviations between the formalisation and the paper.
First, instead of reducing to step graphons and their matrices (\Cref{lem:step-reduction}), we work with integral operators in the formalisation.
The measure $\mu$ is replaced by a finitely supported measure with the moments $g^{\mathsf T}A^jg$ for $j\le m$, which is all the argument uses; this measure comes from the spectral theorem applied to the compression of $A$ to the Krylov subspace spanned by $g,Ag,\ldots,A^mg$.
Second, we prove \Cref{lem:pge23-beta} from a Beta representation of $h_j(a^{\times s},b^{\times t})$ for all $s,t\ge1$, obtained by binomial expansion and termwise Beta integrals; the case $(s,t)=(r+1,r)$ gives the term $h_{n-1}(q^{\times(r+1)},\lambda^{\times r})$.
Third, in the proof of \Cref{lem:pge23-gamma-moment}, $H'$ is computed directly from the moment recurrence, and we prove $L(t)<0$ by splitting $(0,\infty)$ into a few intervals and using tangent-line bounds for the logarithm on each.
Finally, we prove \Cref{prop:main-equality} by applying $q$-regularity and the Cauchy--Schwarz inequality to the cycle densities $\tC nU$, viewed as Hilbert--Schmidt inner products of powers of $\Tker U$; the equality cases give $U^2=U$ almost everywhere and \eqref{eq:equal-neighbourhoods} directly.

\paragraph{Acknowledgements.}
We would like to thank Jineon Baek, Gyeongwon Jeong, Jang Soo Kim, Joonkyung Lee,
Sangho Lim, and Seonghun Park for helpful discussions.
S. Im was supported by a KIAS Individual Grant (AP109501) at the Korea Institute for Advanced Study.
T. Kim and H. Yang were supported by the National Research Foundation of Korea (NRF)
grant funded by the Korean Government (MSIT) (No.~RS-2023-00279680).

\appendix
\crefalias{section}{appendix}

\section{Proof of the step-graphon reduction}
\label{app:step-reduction}

\begin{proof}[Proof of \Cref{lem:step-reduction}]
  For an integer $s\ge1$, set $N=2^s$. Let $I_1,\ldots,I_N$ be the
  intervals of length $1/N$, and
  define
  \[
    W_s(x,y)=N^2\int_{I_a\times I_b}W(u,v)\dd u\dd v
    \qquad ((x,y)\in I_a\times I_b).
  \]
  Each $W_s$ is a symmetric graphon with exactly the same edge density as
  $W$:
  \[
    \int_{[0,1]^2}W_s
    =\sum_{a,b=1}^N\int_{I_a\times I_b}W=p.
  \]
  Moreover,
  \begin{equation}\label{eq:step-L1}
    \norm{W_s-W}_1\longrightarrow0
    \qquad(s\to\infty).
  \end{equation}
  To verify the convergence, let $S$ be a function constant on the
  rectangles obtained by dividing each coordinate interval into
  $2^{s_0}$ equal parts. For $s\ge s_0$, each averaging cell lies in
  one of these rectangles, so its average of $S$ equals $S$ itself.
  On each cell, the absolute value of the average of $W-S$ is at most
  the average of $|W-S|$. Summing over the cells gives
  $\norm{W_s-S}_1\le\norm{W-S}_1$, and hence
  \[
    \norm{W_s-W}_1
    \le \norm{W_s-S}_1+\norm{S-W}_1
    \le2\norm{W-S}_1.
  \]
  Dyadic step functions are dense in $L^1([0,1]^2)$, proving
  \eqref{eq:step-L1}.

  If $F$ is a finite simple graph, subtract the two products defining
  $\dens F{W_s}$ and $\dens FW$ by replacing one edge factor at a time.
  Each resulting term contains one factor $W_s-W$; all other factors
  lie in $[0,1]$. Integrating its absolute value bounds that term by
  $\norm{W_s-W}_1$. There are $|E(F)|$ terms, so
  \[
    \left|\dens F{W_s}-\dens FW\right|
    \le |E(F)|\norm{W_s-W}_1.
  \]
  The same bound holds for $U_s=1-W_s$ and $U=1-W$. Hence all cycle and
  path densities in \eqref{eq:intro-main} and
  \eqref{eq:intro-commonality} converge, while $p$ and $q$
  remain unchanged.
  Applying the two inequalities to each $W_s$ and passing to the limit
  proves the lemma.
\end{proof}

\section{Proof of the odd-to-even Gamma moment bound}
\label{app:gamma-moment}

It remains to prove \Cref{lem:pge23-gamma-moment}. Let $r,j\ge1$ be
integers, let $Y\sim\GammaLaw(r,1)$, and let $z\ge0$. We must show that
\begin{equation}
  3j\,\E(zY-1)^{2j-1}
  \le(r+j)\,\E(zY-1)^{2j}. \label{eq:pge23-gamma-moment2}
\end{equation}

\begin{proof}[Proof of \Cref{lem:pge23-gamma-moment}]
  The case $z=0$ is immediate: the left-hand side equals $-3j<0$, while
  the right-hand side equals $r+j>0$.

  Now assume $z>0$, write $b=1/z$, and define
  \[
    F(b)=\E(Y-b)^{2j},
    \quad\text{so that}\quad
    F'(b)=-2j\,\E(Y-b)^{2j-1},
    \quad
    F''(b)=2j(2j-1)\,\E(Y-b)^{2j-2}.
  \]
  Since $zY-1=z(Y-b)$, dividing \eqref{eq:pge23-gamma-moment2} by
  $z^{2j}>0$ gives the equivalent inequality
  \begin{equation}\label{eq:pge23-gamma-H}
    H(b):=(r+j)F(b)+\frac{3b}{2}F'(b)\ge0.
  \end{equation}

  We first derive a differential equation for $F$ by integration by
  parts against the Gamma density. This equation shows that any maximal interval
  on which $H$ is negative must contain an explicit point $b_*$.
  At that point, another integration by parts reduces positivity to a
  comparison of two local maxima. We complete this comparison with a
  scalar bound that is uniform in $r$ and $j$.

  \smallskip
  \noindent\emph{The differential equation.}

  For every polynomial $f$, integration by parts against the density of
  $Y\sim\GammaLaw(r,1)$ gives the identity below; the boundary term
  $y^re^{-y}f(y)$ vanishes both at $0$ and as $y\to\infty$:
  \begin{equation*}
    \E\bigl[Yf'(Y)\bigr]=\E\bigl[(Y-r)f(Y)\bigr].
  \end{equation*}
  For an integer $k\ge1$, apply this identity to $f(y)=(y-b)^k$
  and write $I_\ell(b)=\E(Y-b)^\ell$ for integers $\ell\ge0$.
  The identity becomes
  \[
    kI_k(b)+kbI_{k-1}(b)=I_{k+1}(b)+(b-r)I_k(b),
  \]
  so rearranging gives the moment recurrence
  \begin{equation*}
    I_{k+1}(b)=(r+k-b)I_k(b)+kbI_{k-1}(b).
  \end{equation*}
  Taking $k=2j-1$ and using $I_{k+1}=F$,
  $I_k=-F'/(2j)$, and $I_{k-1}=F''/(2j(2j-1))$ gives
  \begin{equation}\label{eq:pge23-gamma-ODE}
    bF''(b)+(b-r-2j+1)F'(b)-2jF(b)=0.
  \end{equation}

  \smallskip
  \noindent\emph{Reduction to a single value of $b$.}
  First observe that
  \[
    H(0)=(r+j)\E Y^{2j}>0,
    \qquad
    H(b)=(r+4j)b^{2j}+O(b^{2j-1})
    \qquad(b\to\infty),
  \]
  so $H$ is positive at both ends of $[0,\infty)$.
  If $H$ were negative somewhere, the maximal interval of negativity containing that point
  would have a first zero, where $H'\le0$, and a later zero, where
  $H'\ge0$.

  To locate these endpoints, we next determine the sign of $H'$ at a zero of $H$.
  Differentiating \eqref{eq:pge23-gamma-H} and then eliminating $F''$ by \eqref{eq:pge23-gamma-ODE}, we obtain
  \[
    H'(b)=\left(r+j+\frac{3}{2}\right)F'(b)+\frac{3b}{2}F''(b)
    =3jF(b)+\left(\frac{5r}{2}+4j-\frac{3b}{2}\right)F'(b).
  \]
  Let $b_0>0$ be a zero of $H$, so that $H(b_0)=0$. By \eqref{eq:pge23-gamma-H}, we have
  \[
    F'(b_0)=-\frac{2(r+j)}{3b_0}F(b_0).
  \]
  Define
  \begin{equation*}
    b_*:=\frac{(r+j)(5r+8j)}{3(r+4j)}.
  \end{equation*}
  Substituting the expression for $F'(b_0)$ into the formula for $H'$
  gives
  \begin{equation}\label{eq:pge23-gamma-crossing}
    \frac{H'(b_0)}{F(b_0)}
    =(r+4j)\frac{b_0-b_*}{b_0}.
  \end{equation}
  Since $F(b_0)=\E(Y-b_0)^{2j}>0$, we distinguish three cases using
  \eqref{eq:pge23-gamma-crossing}.
  \begin{enumerate}[label=(\roman*)]
    \item If $b_0<b_*$, then $H'(b_0)<0$.
    \item If $b_0>b_*$, then $H'(b_0)>0$.
    \item If $b_0=b_*$, then $H(b_*)=H(b_0)=0$ and $H'(b_*)=0$.
    Differentiating the formula for $H'$ above and using
    \eqref{eq:pge23-gamma-ODE} gives
    \[
      H''(b_*)=\frac{r+4j}{b_*}F(b_*)>0.
    \]
    Thus $b_*$ is a strict local minimum of $H$ and cannot be an endpoint
    of an interval on which $H$ is negative.
  \end{enumerate}
  The first endpoint of the interval of negativity would therefore lie below $b_*$
  and the second above it,
  so this interval would contain $b_*$ and give $H(b_*)<0$.
  To rule out this possibility, it remains only to prove
  \begin{equation}\label{eq:pge23-gamma-Hbstar}
    H(b_*)\ge0.
  \end{equation}

  \smallskip
  \noindent\emph{An integral representation at $b_*$.}
  We first write $H(b_*)$ as an integral.
  Since $F'(b)=-2j\E(Y-b)^{2j-1}$, \eqref{eq:pge23-gamma-H} gives
  \[
    H(b_*)=(r+j)\E(Y-b_*)^{2j}-3jb_*\E(Y-b_*)^{2j-1}
    =(r+j)\,\E\bigl[(Y-b_*)^{2j-1}(Y-b_*-c)\bigr],
  \]
  where $c=\frac{3jb_*}{r+j}$.
  Expanding the expectation as an integral and substituting $x=y-b_*$,
  \begin{equation}\label{eq:pge23-gamma-H-integral}
    H(b_*) =\frac{r+j}{\Gamma(r) e^{b_*}}\int_{-b_*}^{\infty}
    x^{2j-1}(x-c)(x+b_*)^{r-1}e^{-x}\dd x.
  \end{equation}

  The integrand changes sign, so this representation does not immediately
  establish positivity. We therefore seek a more useful representation by
  integration by parts. To obtain the factor
  $x^{2j-1}(x+b_*)^{r-1}e^{-x}$ upon differentiation, set
  \begin{equation*}
    G(x)=x^{2j}(x+b_*)^re^{-x},
    \qquad -b_*\le x<\infty,
  \end{equation*}
  so that
  \[
    G'(x)=x^{2j-1}(x+b_*)^{r-1}e^{-x}
    \bigl(-x^2+(r+2j-b_*)x+2jb_*\bigr).
  \]
  The choice of $b_*$ makes $c=3jb_*/(r+j)$ a root of this quadratic.
  Thus, the factor $x-c$ in the signed integrand of
  \eqref{eq:pge23-gamma-H-integral} also appears in $G'$, up to sign.
  More precisely, the quadratic factors as $(c-x)(x+d)$, where
  \[
    d=\frac{2jb_*}{c}=\frac{2(r+j)}{3}.
  \]
  A direct calculation gives
  \[
    b_*-d=\frac{r(r+j)}{r+4j}>0,
  \]
  so $-d$ lies in the interval of integration.
  Thus
  \begin{equation}\label{eq:pge23-gamma-Gprime}
    G'(x)=x^{2j-1}(x+b_*)^{r-1}e^{-x}(c-x)(x+d),
  \end{equation}
  and the integrand of \eqref{eq:pge23-gamma-H-integral} is $-G'(x)/(x+d)$, whence
  \[
    \frac{\Gamma(r)e^{b_*}}{r+j}H(b_*)
    =-\int_{-b_*}^{\infty}\frac{G'(x)}{x+d}\dd x.
  \]
  The factor $x+d$ in \eqref{eq:pge23-gamma-Gprime} cancels the denominator,
  so $G'(x)/(x+d)$ extends continuously to $x=-d$.
  To integrate by parts across this point, we use $G(x)-G(-d)$ as a
  primitive of $G'(x)$. Since $G'(-d)=0$, Taylor's theorem gives
  $G(x)-G(-d)=O((x+d)^2)$ as $x\to-d$. Thus
  $(G(x)-G(-d))/(x+d)$ tends to zero, while
  $(G(x)-G(-d))/(x+d)^2$ extends continuously to $x=-d$.

  Integration by parts on either side of $-d$ therefore yields
  \begin{align}
    &-\int_{-b_*}^{\infty}\frac{G'(x)}{x+d}\dd x\notag\\
    &\quad=\left[-\frac{G(x)-G(-d)}{x+d}\right]_{-b_*}^{-d}
    +\left[-\frac{G(x)-G(-d)}{x+d}\right]_{-d}^{\infty}
    +\int_{-b_*}^{\infty}\frac{G(-d)-G(x)}{(x+d)^2}\dd x\notag\\
    &\quad=\frac{G(-d)}{b_*-d}
    +\int_{-b_*}^{\infty}\frac{G(-d)-G(x)}{(x+d)^2}\dd x,
    \label{eq:pge23-gamma-ibp-max}
  \end{align}
  where the evaluations at $-d$ are understood by continuity, and we used
  $G(-b_*)=0$ and $G(x)\to0$ as $x\to\infty$.

  Since $G(-d)=d^{2j}(b_*-d)^re^{d}>0$ and $b_*-d>0$, the term
  $G(-d)/(b_*-d)$ is positive. If $G(x)\le G(-d)$ for every
  $x\ge-b_*$, the integral is nonnegative as well, giving $H(b_*)>0$.
  It therefore suffices to prove that $G(x)\le G(-d)$ for every $x\ge-b_*$.

  \smallskip
  \noindent\emph{Comparison of the two local maxima.}
  By \eqref{eq:pge23-gamma-Gprime}, the two nonzero local maxima occur at
  $-d$ and $c$.
  Their ratio is
  \begin{equation}\label{eq:pge23-gamma-max-ratio}
    \frac{G(c)}{G(-d)}
    =\left(\frac {c}{d}\right)^{2j}
    \left(\frac{b_*+c}{b_*-d}\right)^r e^{-(c+d)}.
  \end{equation}
  Let $t=j/r>0$, and set
  \[
    A(t)=\frac{3t(5+8t)}{2(1+t)(1+4t)},
    \qquad
    B(t)=\frac{(1+4t)(5+8t)}{3(1+t)}.
  \]
  The logarithm of \eqref{eq:pge23-gamma-max-ratio}, divided by $r$, is
  \[
    L(t)=2t\log A(t)+\log B(t)
      -\frac{32t^2+25t+2}{3(1+4t)}.
  \]
  After division by $r$, the parameters $r$ and $j$ occur only through
  their ratio $t$. Proving $L(t)<0$ for every $t>0$ therefore establishes
  the maximum comparison uniformly in $r$ and $j$.

  \smallskip
  \noindent\emph{Proof of the scalar bound $L(t)<0$.}
  Direct differentiation gives
  \begin{align}
    L'(t)
    &=2\log A(t)
    -\frac{(32t^2+16t-7)(32t^2+25t+2)}
    {3(t+1)(4t+1)^2(8t+5)},\label{eq:pge23-gamma-Lprime}\\
    L''(t)
    &=-\frac{(32t^2+25t+2)
    (128t^4+224t^3+120t^2-28t-25)}
    {t(t+1)^2(4t+1)^3(8t+5)^2}.
    \label{eq:pge23-gamma-Lsecond}
  \end{align}
  In \eqref{eq:pge23-gamma-Lsecond}, the denominator and the factor
  $32t^2+25t+2$ are positive for $t>0$. The initial minus sign therefore
  makes the sign of $L''$ opposite to that of the quartic
  $P(t)=128t^4+224t^3+120t^2-28t-25$.
  The coefficients of $P$ change sign once, so by Descartes' rule of signs $P$ has at most one positive root.
  Since $P(0)=-25<0$ and $P(t)\to\infty$, it has exactly one positive
  root $t_0$. Consequently, $L'$ increases on $(0,t_0)$ and decreases on
  $(t_0,\infty)$.

  We now determine the sign of $L'$ near $0$, at $1$, and at $3/2$
  to locate its two zeros.
  We will use the elementary bounds $\log2<7/10$ and
  $\log13<13/5$. The positive terms in the exponential series give
  \[
    e^{7/10}>\sum_{\ell=0}^{3}\frac{(7/10)^\ell}{\ell!}>2,
    \qquad
    e^{13/5}>\sum_{\ell=0}^{6}\frac{(13/5)^\ell}{\ell!}>13,
  \]
  which prove both bounds by taking logarithms.
  As $t\downarrow0$ the argument $A(t)$ tends to $0$ while the rational term of \eqref{eq:pge23-gamma-Lprime} tends to $-14/15$, so $L'(t)\to-\infty$.
  At $t=1$ we have $A(1)=39/20$ and the rational term equals $2419/1950$, so the bound
  \[
    \log x\ge\frac{2(x-1)}{x+1}\qquad(x\ge1)
  \]
  gives $2\log A(1)\ge\frac{76}{59}$ and
  \[
    L'(1)\ge\frac{76}{59}-\frac{2419}{1950}
    =\frac{5479}{115050}>0.
  \]
  At $t=3/2$ we have $A(3/2)=153/70$ and the rational term equals $19847/12495$; concavity of the logarithm at $x=2$ gives
  \[
    \log\frac{153}{70}\le\log2+\frac{1}{2}\left(\frac{153}{70}-2\right)
    =\log2+\frac{13}{140},
  \]
  so with $2 \log 2 < 7/5$,
  \[
    L'\left(\frac{3}{2}\right)
    <\frac{7}{5} + \frac{13}{70}-\frac{19847}{12495}
    =-\frac{67}{24990}<0.
  \]

  Since $L'$ first increases and then decreases, these signs show that it
  has exactly two roots: some $t_1\in(0,1)$ with $L''(t_1)>0$, and some
  $t_2\in(1,3/2)$ with $L''(t_2)<0$.
  Thus, $L$ decreases on $(0,t_1)$, increases on $(t_1,t_2)$, and
  decreases on $(t_2,\infty)$. Its supremum over $(0,\infty)$ is therefore
  the larger of $L(t_2)$ and its limit at $0$.
  The latter is negative,
  \[
    \lim_{t\downarrow0}L(t)=\log\frac{5}{3}-\frac{2}{3}<0,
  \]
  and $t_2$ lies in $[1,3/2]$, so it remains to control $L$ on $[1,3/2]$.
  By concavity of $\log$,
  \[
    \log A\le\log2+\frac{A-2}{2},
    \qquad
    \log B\le\log13+\frac{B-13}{13}.
  \]
  Using $\log2<7/10$ and $\log13<13/5$, exact simplification gives
  \begin{equation}\label{eq:pge23-gamma-L-rational}
    L(t)<
    \frac{3(96t^3-191t^2-16t+46)}
    {130(t+1)(4t+1)}.
  \end{equation}
  On $[1,3/2]$ we have $96t-191<0$ and $t^2\ge1$, so the cubic numerator is at most $(96t-191)-16t+46=80t-145\le-25$, and $L(t)<0$ there as well.

  Together with the preceding description of the supremum, this proves
  $L(t)<0$ for every $t>0$.

  \smallskip
  \noindent\emph{Completion of the moment bound.}
  The ratio \eqref{eq:pge23-gamma-max-ratio} is therefore smaller than $1$,
  so $G$ attains its maximum at $-d$. Thus, the integrand in
  \eqref{eq:pge23-gamma-ibp-max} is nonnegative.
  The right-hand side of \eqref{eq:pge23-gamma-ibp-max} is then positive,
  giving $H(b_*)>0$. The earlier analysis of the zeros of $H$ shows that
  if $H$ were negative anywhere, then $H(b_*)<0$. Thus, $H(b)\ge0$ for
  every $b>0$, which is \eqref{eq:pge23-gamma-H} and hence implies
  \eqref{eq:pge23-gamma-moment2}.
\end{proof}

\bibliographystyle{abbrv}
\bibliography{references}

\end{document}